\documentclass[11pt,reqno]{amsart}
\usepackage[utf8]{inputenc} 
\usepackage[T1]{fontenc}    
\usepackage[english]{babel} 
\usepackage{amsmath,amsthm,amsfonts,amssymb,amsbsy,amscd}
\usepackage{mathrsfs,float}
\usepackage[mathcal]{euscript}
\usepackage{enumitem,mathtools}
\usepackage[colorlinks=true,linkcolor=blue,citecolor=green]{hyperref}
\usepackage{tikz}
\usepackage{tikz-cd}
\usepackage{xcolor}

\numberwithin{equation}{section}

\usepackage[a4paper,text={6.1in,8in},centering,top=1.5in]{geometry}
\small\normalsize
\title[Limit theorems for random matrix products]{Limit theorems for random matrix products with a simple top Lyapunov exponent}
\author[Barrientos]{Pablo G. Barrientos}
\address{Instituto de Matem\'atica e Estat\'istica, UFF, Niter\'oi, Brazil}
    \email{pgbarrientos@id.uff.br}

 \author[Malicet]{Dominique Malicet}
\address{LAMA, Universit\'e Gustave Eiffel, Champs-Sur-Marne, France} 
\email{dominique.malicet@univ-eiffel.fr}
   
\date{}

\newtheorem{mainthm}{Theorem}

\newtheorem{theorem}{Theorem}[section]
\newtheorem{proposition}[theorem]{Proposition}
\newtheorem{lemma}[theorem]{Lemma}
\newtheorem{corollary}[theorem]{Corollary}
\newtheorem{definition}[theorem]{Definition}
\newtheorem{remark}[theorem]{Remark}

\theoremstyle{definition}
\newtheorem{rem}[theorem]{Remark}

\newcommand{\R}{\mathbb R}
\newcommand{\N}{\mathbb N}
\newcommand{\C}{\mathbb C}
\newcommand{\Pp}{\mathbb P}
\newcommand{\1}{\mathbf 1}
\newcommand{\cI}{\mathcal I}

\newcommand{\cN}{\mathcal N}
\newcommand{\Hol}{\operatorname{Hol}}
\newcommand{\Lip}{\operatorname{Lip}}
\newcommand{\supp}{\operatorname{supp}}
\newcommand{\eqdef}{\mathrel{\mathop:}=}

\DeclareMathOperator{\Var}{Var}
\DeclareMathOperator{\Ree}{Re}

\begin{document}

\begin{abstract}
We study products of independent, identically distributed random matrices
with an exponential moment and a simple top Lyapunov exponent, without
assuming irreducibility. Outside an exceptional invariant subspace called
the equator, we establish contraction on average for a weighted projective
distance and a spectral gap for the associated Markov operator. We obtain
a central limit theorem, a Berry--Esseen bound, deviation estimates near
the mean, and an almost sure invariance principle both for the growth of
a fixed vector outside the equator and for the operator norm. For fixed
vectors, a non-arithmeticity assumption also yields a local limit theorem
and a renewal theorem. Our contraction criterion applies more generally
to mostly contracting random maps on locally compact spaces with a
Lyapunov function, including certain random interval maps.
\end{abstract}

\maketitle \thispagestyle{empty}
\section{Introduction}\label{s:intro}

Let $(A_k)_{k\geq0}$ be independent, identically distributed random
matrices in $\operatorname{GL}(m,\R)$, $m\geq2$, on a probability space
$(\Omega,\mathscr F,\Pp)$, and put
$A(\omega)=A_0(\omega)$ and
$A^n(\omega)=A_{n-1}(\omega)\cdots A_0(\omega)$ for $n\geq1$,
with $A^0(\omega)=I$.
We fix a Euclidean norm on $\R^m$ and use the associated operator norms.
Our interest is in the asymptotic behavior of $\log\|A^n(\omega)\|$
and $\log\|A^n(\omega)x\|$ for fixed $x\neq0$.
Throughout the matrix results, we assume an exponential moment:
\begin{equation}\label{eq:exp-moment}
 \int\max\{\|A(\omega)\|,\|A(\omega)^{-1}\|\}^{\eta_0}
 \,d\Pp<\infty
 \qquad\text{for some }\eta_0>0.
\end{equation}
The Lyapunov exponents can be defined through exterior powers
\cite{BouLac:85}. In particular,
\begin{align*}
 \lambda_1(A) &=\lim_{n\to \infty}\frac1n
       \int\log\|A^n(\omega)\|\,d\Pp,
 \\ \lambda_1(A)+\lambda_2(A)  &=\lim_{n\to \infty}\frac1n
       \int\log\|\wedge^2 A^n(\omega)\|\,d\Pp.
\end{align*}
The classical setting assumes that the semigroup generated by the support
of the law of $A$ is strongly irreducible and proximal. This implies
$\lambda_1(A)>\lambda_2(A)$~\cite{GR85} and yields numerous
limit theorems, beginning with Le Page's seminal
work~\cite{LePage}. We prove that these limit theorems still hold
under the sole condition $\lambda_1(A)>\lambda_2(A)$. See Theorems~\ref{main:norm-clt}, ~\ref{mainthm:pointwise-CLT-BE} and~\ref{mainthm:linear-asip-final}.

Le Page's method starts from uniform contraction on average of the
projective action, derives a spectral gap for its Markov operator, and
then obtains limit theorems by the Nagaev method. Without irreducibility,
an invariant subspace $E$ called the equator may repel nearby projective
points, so uniform contraction need not hold. By weighting the
projective distance near this subspace, we obtain contraction on
average outside it (Theorem~\ref{main:spectral-gap-linear}). This
yields a spectral gap on a suitable function space and the corresponding
quantitative limit theorems. Our spectral gap strengthens the result of~\cite{AG20}, which, under
the same assumptions, establishes a unique stationary probability $\nu$
outside the projectivization of $E$ and exponential convergence in law toward it from
every initial direction outside the equator. 

We begin with the asymptotic behavior of $\log\|A^n(\omega)\|$.
Almost surely, its quotient by $n$ tends to $\lambda_1(A)$.
When the top Lyapunov exponent is
simple, we obtain a central limit theorem with a convergence rate and
a Hoeffding-type bound for deviations from $\lambda_1(A)$.

\begin{mainthm}\label{main:norm-clt}
Assume~\eqref{eq:exp-moment} and $\lambda_1(A)>\lambda_2(A)$.
The following hold for some $\sigma\geq0$.
\begin{enumerate}[label=\textup{(\roman*)},itemsep=0.25cm]
\item \emph{Central limit theorem:} 
\[
 \frac{1}{\sqrt{n}} \big(\log\|A^n(\omega)\|-n\lambda_1(A) \big)
\xrightarrow{\,\mathrm{law}\,}\cN(0,\sigma^2) \quad \text{as $n\to \infty$}.
\]
\item \emph{Berry--Esseen bound:} if $\sigma>0$, then there is $C>0$ such
that, for all $n\geq1$,
\begin{equation}\label{eq:norm-BE}
 \sup_{u\in\R}\left|
 \Pp\left(\frac{\log\|A^n(\omega)\|-n\lambda_1(A)}
                   {\sigma\sqrt n}\leq u\right)-\frac1{\sqrt{2\pi}}\int_{-\infty}^u e^{-s^2/2}\,ds
 \right|\leq\frac C{\sqrt n}.
\end{equation}
\item \emph{Large deviations:} there are
$\varepsilon_0,h,C>0$ such that, for all $n\geq1$ and
$0<\varepsilon<\varepsilon_0$,
\begin{equation}\label{eq:LDa}
 \Pp\left(\left|\frac1n\log\|A^n(\omega)\|-\lambda_1(A)\right|
 >\varepsilon\right)\leq Ce^{-hn\varepsilon^2}.
\end{equation}
\end{enumerate}
\end{mainthm}

Without simplicity of the top exponent, the norm CLT can fail:
Furstenberg--Kesten~\cite[Ex.~2]{FurstenbergKestenCLT} and
Benoist--Quint~\cite[Ex.~4.15]{BenoistQuint} give examples of folded
Gaussian limits. For $m=2$, the CLT and Berry--Esseen bound were
established in~\cite[Props.~VIII and X]{BMbook}.

We next consider $\log\|A^n(\omega)x\|$ for a fixed unit vector $x$.
Its growth rate is $\lambda_1(A)$ for $\Pp$-almost every~$\omega$,
except possibly when $x$ belongs to an exceptional deterministic
$A$-invariant subspace $E$, called the \emph{equator} of $A$,~\cite{FK83}.
Equivalently, $E$ is the largest $A$-invariant subspace of $\R^m$
such that $\lambda_1(A|_E)<\lambda_1(A)$, with the convention
$\lambda_1(A|_{\{0\}})=-\infty$. It is always a proper subspace and
may equal $\{0\}$, in which case $A$ is \emph{quasi-irreducible}.
We refer to~\cite{AG20,BMbook} for details. For $x\notin E$, we obtain a CLT, a Berry--Esseen bound, a local
central limit theorem, and a renewal theorem for
$\log\|A^n(\omega)x\|$. The last two require a non-arithmeticity
condition. 	

Under the simplicity assumption of the first Lyapunov exponent, we say that $A$ is
	\emph{non-arithmetic} if there are no $a\in\R$, $b>0$, and measurable
	real-valued projective function $g$, defined $\nu$-almost everywhere, such that
	\[
	\phi_\omega(\hat x)+g(A(\omega)\hat x)-g(\hat x)
	\in a+b\mathbb Z
	\quad\text{for $(\Pp\times\nu)$-almost every~$(\omega,\hat x)$.}
	\]
Here, for each $\omega \in\Omega$, 
	\[
	\phi_\omega(\hat x)=\log\frac{\|A(\omega)x\|}{\|x\|}
	\]
where $\hat x$ denotes the line induced by a nonzero vector $x$ and $\nu$ is the unique stationary probability measure outside of the projectivization of the equator. Under~\eqref{eq:exp-moment}, non-arithmeticity implies that the
asymptotic variance is positive. Indeed, the zero-variance
coboundary criterion~\cite[Thm.~A(iii)]{HH01} gives, if $\sigma^2=0$, a
measurable real-valued projective function $g$ such that $\phi_\omega(\hat x)+g(A(\omega)\hat x)-g(\hat x)
 =\lambda_1(A)$ for $(\Pp\times\nu)$-almost every~$(\omega,\hat x)$. This implies arithmeticity with $a=\lambda_1(A)$ and any $b>0$.

\begin{mainthm}\label{mainthm:pointwise-CLT-BE}
Assume~\eqref{eq:exp-moment},  $\lambda_1(A)>\lambda_2(A)$ and let $\sigma$ be as in Theorem~\ref{main:norm-clt}. Then for every unit vector $x\in\R^m\setminus E$, 
\begin{enumerate}[label=\textup{(\roman*)}]
\item \emph{Central limit theorem:}
\[
 \frac{1}{\sqrt{n}}\big(\log \| A^n(\omega)x\|-n\lambda_1(A) \big)
 \xrightarrow[n\to \infty]{\mathrm{law}}\cN(0,\sigma^2).
\]
\item \emph{Berry--Esseen bound:} if $\sigma>0$, there is a constant $C_x>0$ such that   for all $n\geq1$,
\[
 \sup_{u\in\R}\left|
 \Pp\left(\frac{\log \| A^n(\omega)x\|-n\lambda_1(A)}{\sigma\sqrt n}\leq u\right)
 -\frac1{\sqrt{2\pi}}\int_{-\infty}^u e^{-s^2/2}\,ds \right|\leq\frac{C_x}{\sqrt n}.
\]
\item \emph{Local central limit theorem:} If $A$ is
non-arithmetic, then for every $g\in C_c(\R)$,
\[ \qquad \quad
 \sup_{a\in\R}\left|
 \sigma\sqrt{2\pi n}\int g(\log \| A^n(\omega)x\|-n\lambda_1(A)-a)\,d\Pp
 -e^{-a^2/(2n\sigma^2)}\int g(s)\,ds
 \right|\longrightarrow0.
\]
\item \emph{Renewal theorem:} If $\lambda_1(A)>0$ and $A$ is
non-arithmetic, then, for every $g\in C_c(\R)$,
\[
 \sum_{n\geq1}\int g(\log \| A^n(\omega)x\|-a)\,d\Pp =
 \frac{\1_{\{a>0\}}}{\lambda_1(A)}\int g(s)\,ds+o(1),
\quad  a\to \pm\infty.
\]
\item \emph{Large deviations:} There are
$\varepsilon_0,h>0$ and $C_x>0$ such that, for all $n\geq1$ and
$0<\varepsilon<\varepsilon_0$,
\[
 \Pp\left(\left|\frac{1}{n}\log \|A^n(\omega)x\|-\lambda_1(A)\right|>\varepsilon\right)
 \leq C_xe^{-hn\varepsilon^2}.
\]
If $\sigma>0$, there is a real-analytic
rate function $c$ on $(-\varepsilon_0,\varepsilon_0)$, independent
of $x$, such that
$c(0)=c'(0)=0$, $c''(0)=1/\sigma^2$, and, for
$0<\varepsilon<\varepsilon_0$,
\[
 \begin{aligned}
 \lim_{n\to \infty}\frac1n\log\Pp\left(
       \frac{1}{n}\log \|A^n(\omega)x\|-\lambda_1(A)>\varepsilon\right)&=-c(\varepsilon),\\
 \lim_{n\to \infty}\frac1n\log\Pp\left(
       \frac{1}{n}\log \|A^n(\omega)x\|-\lambda_1(A)<-\varepsilon\right)&=-c(-\varepsilon).
 \end{aligned}
\]
\end{enumerate}
\end{mainthm}

We now state the contraction and spectral gap estimates underlying
these results. Let $\hat x$ denote the line through a nonzero vector
$x$, and let $P(\R^m)=\{\hat x:x\in\R^m\setminus\{0\}\}$ be the
projective space of $\R^m$, equipped with the angular distance
$d(\hat x,\hat y)=\|x\wedge y\|/(\|x\|\,\|y\|)$.

Let $A$ be a random matrix. The projective maps and Markov operator 
are defined by
\[
f_\omega(\hat x)=A(\omega)\hat x,\qquad
(P\varphi)(\hat x)=\int\varphi(f_\omega(\hat x))\,d\Pp.
\]
Thus $f_\omega^n(\hat x)=A^n(\omega)\hat x$ for $n\geq0$.
On $X=P(\R^m)\setminus P(E)$, define
$\delta(\hat x)=d(\hat x,P(E))$ if $E\neq\{0\}$ and
$\delta\equiv1$ otherwise. For $\alpha,\beta>0$ define 
\[
 \Delta_{\alpha,\beta}(\hat x,\hat y)
 \eqdef\frac{d(\hat x,\hat y)^\alpha}
 {\delta(\hat x)^\beta\delta(\hat y)^\beta},
 \qquad
 [\varphi]_{\alpha,\beta}
 \eqdef\sup_{\hat x\neq\hat y}
 \frac{|\varphi(\hat x)-\varphi(\hat y)|}
      {\Delta_{\alpha,\beta}(\hat x,\hat y)}.
\]
The function $\Delta_{\alpha,\beta}$ need not be a metric.
For a fixed $\hat x_0\in X$, let
\[
 \Hol_{\alpha,\beta}(X)
 \eqdef\{\varphi\in C(X;\C):
                         [\varphi]_{\alpha,\beta}<\infty\},
 \qquad
 \|\varphi\|_{\alpha,\beta}
 \eqdef|\varphi(\hat x_0)|+[\varphi]_{\alpha,\beta}.
\]
This norm makes $\Hol_{\alpha,\beta}(X)$ a Banach space;
the choice of $\hat x_0$ changes it only by an equivalent norm.

\begin{mainthm}\label{main:spectral-gap-linear}
Assume~\eqref{eq:exp-moment} and
$\lambda_1(A)>\lambda_2(A)$. There is $\varepsilon_0>0$ such
that, whenever $\alpha,\beta>0$ satisfy
$\beta<\varepsilon_0$ and $\alpha/\beta<\varepsilon_0$, there are
$C>0$ and $\rho\in(0,1)$ for which
\[
 \int\Delta_{\alpha,\beta}
 \bigl(f_\omega^n(\hat x),f_\omega^n(\hat y)\bigr)
 \,d\Pp
 \leq C\rho^n\Delta_{\alpha,\beta}(\hat x,\hat y)
 \quad \text{for $\hat x,\hat y\in X,\ n\geq0$}.
\]
Moreover, $P$ has a unique stationary probability $\nu$ on $X$, and has a spectral gap on $\Hol_{\alpha,\beta}(X)$, that is, 
possibly changing $C$ and $\rho$,
\[
 \left\|P^n\varphi-\int\varphi\,d\nu
 \right\|_{\alpha,\beta}
 \leq C\rho^n\|\varphi\|_{\alpha,\beta}
 \quad \text{for $\varphi\in\Hol_{\alpha,\beta}(X),\ n\geq0$.}
\]
\end{mainthm}

The
contraction also yields a strong approximation for the logarithmic
cocycle, through the almost sure invariance principle of
Cuny--Dedecker--Merlev\`ede~\cite{CDM24}.

\begin{mainthm}\label{mainthm:linear-asip-final}
Assume~\eqref{eq:exp-moment} and
$\lambda_1(A)>\lambda_2(A)$. Let $\sigma^2$ be the  variance
in Theorems~\ref{main:norm-clt} and~\ref{mainthm:pointwise-CLT-BE}. For every unit vector
$x\in\R^m\setminus E$, the process
$\log \|A^n(\omega)x\|-n\lambda_1(A)$ can be realized on an enriched probability space together with independent
$G_k\sim\cN(0,\sigma^2)$ such that
\begin{equation}\label{eq:linear-asip-final-log3}
 \log \|A^n(\omega)x\|
 -n\lambda_1(A)-\sum_{k=1}^nG_k
 =O((\log n)^3)\qquad\text{a.s.}
\end{equation}
The same conclusion holds with $\log\|A^n(\omega)\|$ in place of
$\log \|A^n(\omega)x\|$. If the law of $A$ has compact support in
$\operatorname{GL}(m,\R)$, the error improves to
$O((\log n)^2)$ in both statements.
\end{mainthm}

In particular, the ASIP gives the \emph{functional CLT} and the \emph{law of the
iterated logarithm}, both for a fixed direction outside $E$ and for
the operator norm.

Under strong irreducibility and proximality, $\sigma>0$ always holds~\cite{LePage}. In dimension $m\geq2$, these
assumptions also imply non-arithmeticity
\cite[Prop.~4.6]{GuivarchUrban05}. In the present setting, $\sigma$ can vanish.  The proposition below gives an algebraic characterization of $\sigma=0$. It would be interesting to find an algebraic characterization of
non-arithmeticity under the weaker assumption $\lambda_1>\lambda_2$.

Let $\tilde V=\R^m/E$ and let $\tilde A$ be the induced cocycle.
Its \emph{pinnacle subspace} $\tilde F$ is the intersection of all
$\tilde A$-invariant subspaces $W\subset\tilde V$ such that
$\lambda_1(\tilde A|_W)=\lambda_1(A)$.

\begin{proposition}\label{prop:linear-zero-variance}
Assume~\eqref{eq:exp-moment} and $\lambda_1(A)>\lambda_2(A)$.
The common variance $\sigma^2$ of Theorems~\ref{main:norm-clt} and
\ref{mainthm:pointwise-CLT-BE} vanishes if and only if $\dim\tilde F=1$
and $\|\tilde A(\omega)|_{\tilde F}\|$ is $\Pp$-almost surely constant.
\end{proposition}

\subsection*{Scheme of the proof}
The matrix product acts on the locally compact space
$X=P(\R^m)\setminus P(E)$. When $E\neq\{0\}$, comparison of the
growth on $E$ and on $\R^m/E$ shows that $P(E)$ is a repeller. A
negative power of the distance to $P(E)$ is therefore a Lyapunov
function for a suitable iterate. On the other hand,
$\lambda_1>\lambda_2$ makes the local projective Lyapunov exponent
negative for every stationary measure on $X$. Projective
trajectories also synchronize. In the abstract setting developed
below, a local uniform Kingman theorem turns this negativity into
contraction on compact subsets. The Lyapunov function controls the
dynamics near infinity, and a corrector combines the two estimates
into global contraction for a weighted projective distance. This
gives the spectral gap. Limit theorems for $\log\|A^n x\|$ then follow
from the logarithmic cocycle; those for $\log\|A^n\|$ follow by
applying the pointwise results to a suitable vector for the tensor
square $A\otimes A$.

\subsection*{Organization of the paper}
Sections~\ref{sec:drift} and~\ref{sec:subadditive} collect the
Lyapunov and subadditive tools. Section~\ref{sec:mostly} establishes
contraction and spectral gap results for mostly contracting random
maps on locally compact spaces. Section~\ref{sec:repellers} constructs
Lyapunov functions from compact repellers, while
Section~\ref{sec:limit-theorems} derives abstract limit theorems and
an almost sure invariance principle. Finally,
Section~\ref{s:linear} applies these results to matrix products, and
Section~\ref{s:interval} gives nonlinear examples on the interval.

\section{Preliminaries}
\label{sec:drift}
Throughout this section, $X$ is a locally compact separable metric
space, hence a Polish space, and all probabilities are Borel (and
therefore Radon). We write $\mathscr B$ for the Borel $\sigma$-algebra
and $C_b(X)$ for the bounded continuous real-valued functions.
Weak$^*$ convergence means $\int\varphi\,d\mu_j\to\int\varphi\,d\mu$
for every $\varphi\in C_b(X)$.

\subsection{Markov operators and Lyapunov functions}
A transition probability $P(x,B)$ is a function from $X \times \mathscr{B}$ to $[0,1]$ such that \( P(x,\cdot) \) is a probability measure for all \( x \in X \), and \( P(\cdot,B) \) is a measurable function for all \( B \in \mathscr{B} \). This transition probability induces the Markov operator $P\varphi(x)=\int\varphi(y)P(x,dy)$ on bounded Borel functions. We use the same notation $P\varphi$ for nonnegative Borel functions (possibly taking the value $+\infty$), and for Borel functions bounded above (possibly taking the value $-\infty$). It is positive, $P\1=\1$, and $\|P\varphi\|_\infty\leq\|\varphi\|_\infty$. We assume that $x\mapsto P(x,\cdot)$ is continuous for the topology of weak convergence of probabilities, equivalently $P(C_b(X))\subset C_b(X)$. Such an operator is called \emph{Markov--Feller}. Its dual $P^*$ is characterized by $\int \varphi\,d(P^*\mu)=\int P\varphi\,d\mu$ for any $\varphi\in C_b(X)$. 
A probability $\mu$ is \emph{$P$-stationary} if $P^*\mu=\mu$. We denote by $\cI(P)$ the closed convex set of stationary probabilities; its extreme points are the ergodic stationary probabilities, and we denote by $\cI_{\rm erg}(P)$ the set of such measures. 

\begin{definition}
A \emph{Lyapunov function} for $P$ is a continuous proper function $V:X\to[1,\infty)$ such that
\begin{equation}\label{eq:geometric-drift}
        PV\leq qV+C
\end{equation}
for some constants $q\in(0,1)$ and $C>0$. Here, proper means that for every $R\geq1$ the sublevel set $\{V\leq R\}$ is compact.
\end{definition}

If $X$ is compact, the constant function $V\equiv1$ is a Lyapunov function. Next,  we list some basic properties of Lyapunov functions. 

\begin{proposition}\label{LyapunovFunctionsProperties}
Let $V$ be a Lyapunov function for $P$ and  set $C_V=C/(1-q) +1$. Then, for every $q_V\in(q,1)$, and $n\geq0$, there exists a compact set $K_n\subset X$
    such that
\begin{enumerate}
	\item $P^nV\leq q_V^nV+C_V$ on $X$.
	\item $\frac{P^nV}{V}\leq C_V$ on $X$.
	\item $\frac{P^nV}{V}\leq q_V^n$ outside of $K_n$.
\end{enumerate}
\end{proposition}
\begin{proof}
Fix any $q_V\in (q,1)$. Iterate~\eqref{eq:geometric-drift} to obtain
$P^nV\leq q^nV+\frac{C}{1-q}$.
Since $V\geq1$ and $q<1$, this gives the first two bounds.  For $n\geq1$, take
$K_n=\{V\leq C_V/(q_V^n-q^n)\}$. Then $P^nV/V<q_V^n$
outside of $K_n$.
\end{proof}

A Lyapunov function also gives existence and tightness of stationary measures by the Krylov--Bogolyubov procedure (cf.~\cite[Cor.~4.23]{BenaimHurth}). Recall that a family \(\mathcal M\) of probability measures on \(X\) is
\emph{tight} if, for every \(\varepsilon>0\), there is a compact
set \(K\subset X\) such that
$ \mu(K)\geq1-\varepsilon$  for every $\mu\in\mathcal M$. Since \(X\) is Polish, tightness is equivalent to relative
compactness in the weak$^*$ topology.

\begin{proposition}\label{prop:tightness}
Let \(P\) be a Markov--Feller operator admitting a Lyapunov
function \(V\). Then \(\cI(P)\) is nonempty and weak$^*$ compact,
and every \(\mu\in\cI(P)\) satisfies
$\int V\,d\mu\leq C_V$. Moreover, for every sequence \((x_n)_{n\geq1}\) contained in a compact
subset of \(X\), the Ces\`aro averages
\[
        \eta_n\eqdef\frac1n\sum_{k=0}^{n-1}(P^k)^*\delta_{x_n}
\]
are tight, and all their weak$^*$ cluster values belong
to \(\cI(P)\). Furthermore, for every \(x\in X\), if \((X_k)_{k\geq0}\) is
a Markov chain with transition operator \(P\) and \(X_0=x\),
the empirical measures
\[
        e_n\eqdef\frac1n\sum_{k=0}^{n-1}\delta_{X_k}
\]
are almost surely tight.
\end{proposition}

\subsection{Some elementary probabilistic tools}
We record a few elementary tools on an arbitrary probability space
$(\Omega,\mathscr F,\Pp)$. For any extended real-valued function $h$, write
$h^+\eqdef\max\{h,0\}$.

\begin{lemma}\label{lem:one-sided-small-powers}
	Let \(Z\) be a random variable taking values in
	\([-\infty,\infty)\). Assume that
	$\int Z\,d\Pp<0$ and $\int e^{\delta Z^+}\,d\Pp<\infty$
	for some $\delta>0$. Then there is $s_0>0$ such that
	$$\int e^{sZ}\,d\Pp<1 \quad \text{ for every $s\in(0,s_0)$.}$$
\end{lemma}

\begin{proof}
	If $\int e^{\delta|Z|}\,d\Pp<\infty$, then
	$\int Z^2e^{\delta|Z|/2}\,d\Pp<\infty$. Taylor's inequality
	$e^t\leq1+t+\tfrac12t^2e^{|t|}$ gives, for small $s>0$,
	$\int e^{sZ}\,d\Pp\leq1+s\int Z\,d\Pp+O(s^2)<1$.
	In the general case, set $Z_M=\max\{Z,-M\}$.
	Since $Z^+$ is integrable and $Z_M\downarrow Z$,
	$\int Z_M\,d\Pp\downarrow\int Z\,d\Pp<0$.
	Choose $M$ so that $\int Z_M\,d\Pp<0$.
	Applying the first case to $Z_M$ yields
	$\int e^{sZ}\,d\Pp\leq\int e^{sZ_M}\,d\Pp<1$
	for all sufficiently small $s>0$.
\end{proof}

\begin{lemma}\label{FatouReverse}
	Let $(Z_n)_{n\geq1}$ be a sequence of random variables taking values in $[-\infty,\infty)$ such that, for some $p>1$,
	$\sup_{n\geq1}\int(Z_n^+)^p\,d\Pp<\infty$. Then
	\[
	\limsup_{n\to \infty}\int Z_n\,d\Pp
	\leq\int\limsup_{n\to \infty}Z_n\,d\Pp.
	\]
\end{lemma}

\begin{proof}
	Set $C\eqdef \sup_{n\geq1}\int(Z_n^+)^p\,d\Pp$. For $M>0$, let
	$Z_n^{(M)}\eqdef \min\{Z_n,M\}$. The pointwise bound
	$Z_n\leq Z_n^{(M)}+Z_n^+\1_{\{Z_n>M\}}$ and
	$\int Z_n^+\1_{\{Z_n>M\}}\,d\Pp\leq C/M^{p-1}$ give
	\[
	\limsup_{n\to \infty}\int Z_n\,d\Pp
	\leq\limsup_{n\to \infty}\int Z_n^{(M)}\,d\Pp+\frac{C}{M^{p-1}}.
	\]
	Fatou's lemma applied to $M-Z_n^{(M)}\geq0$ gives
	\[
	\limsup_{n\to \infty}\int Z_n^{(M)}\,d\Pp
	\leq\int\limsup_{n\to \infty} Z_n^{(M)}\,d\Pp
	\leq\int\limsup_{n\to \infty} Z_n\,d\Pp.
	\]
	Combining the two inequalities and letting $M\to \infty$ gives the result.
\end{proof}

For the next corollary, $X$ is a metric space. We use the sequential characterization of upper semicontinuity.
\begin{corollary}\label{cor:semicontinuity}
	Let $\phi:X\times\Omega\longrightarrow[-\infty,\infty)$ be a random map which is upper semicontinuous in $x$ for $\Pp$-a.e.~$\omega\in\Omega$, and such that, for some $p>1$, $\sup_{x\in X}\int\phi^+(x,\omega)^p\,d\Pp<\infty$. Then the function $x\mapsto\int\phi(x,\omega)\,d\Pp$ is also upper semicontinuous on $X$.
\end{corollary}

\begin{proof}
Apply Lemma~\ref{FatouReverse} along any sequence $x_j\to x$.
\end{proof}

\section{Subadditive sequences}
\label{sec:subadditive}
In this section, $X$ is a metric space and $P$ is a Markov operator on $X$. A sequence $(\phi_n)_{n\geq1}$ of Borel functions $X\to[-\infty,\infty)$ is \emph{$P$-subadditive} if 
$\phi_{n+m}\leq\phi_m+P^m\phi_n$ for all $n,m\geq1$.
In particular, if $\phi_1$ is bounded above, then the functions
$\phi_n/n$ are uniformly bounded above.

\begin{definition}\label{assymp-neg}
A $P$-subadditive sequence $(\phi_n)_{n\geq1}$ is \emph{asymptotically negative}
if, for every compact $K\subset X$,
\[
\limsup_{n\to \infty}\frac1n\sup_{x\in K}\phi_n(x)<0.
\]
\end{definition}

\subsection{A local uniform Kingman theorem}
When $X$ is compact, the next criterion is the uniform subadditive
theorem of~\cite[Thm.~C]{BMbook}. A Lyapunov function allows us to
recover its conclusion on each compact subset of a noncompact space.

\begin{proposition}\label{Kingman-noncompact}
Suppose that $X$ is locally compact, $P$ is Markov--Feller and admits a Lyapunov function.
Let $(\phi_n)_{n\geq1}$ be a $P$-subadditive sequence of upper semicontinuous
functions, with $\phi_1$ bounded above. Then $(\phi_n)_{n\geq1}$ is
asymptotically negative if and only if
\[
\inf_{n\geq1}\int\phi_n\,d\mu<0
\qquad\text{for every }\mu\in\cI(P).
\]
\end{proposition}


\begin{proof}
If $(\phi_n)_{n\geq1}$ is asymptotically negative, then
$\limsup_{n\to \infty}\phi_n(x)/n<0$ for every $x$. Choose a finite upper bound
$M$ for $\phi_1$; then $\phi_n/n\leq M$. Reverse Fatou shows that, for any
$\mu\in\cI(P)$,
\[
\limsup_{n\to \infty}\int\frac{\phi_n}{n}\,d\mu
\leq\int\limsup_{n\to \infty}\frac{\phi_n}{n}\,d\mu<0.
\]
This proves the first implication.

Conversely, assume the condition on stationary measures. For every $\mu\in \cI(P)$ there exists $m\geq1$ such that $\int\phi_m\,d\mu<0$. Using upper semicontinuity, the inequality still holds on a neighborhood of $\mu$ 
in $\cI(P)$. Since $\cI(P)$ is compact by Proposition~\ref{prop:tightness}, taking a finite covering of  $\cI(P)$ and a common multiple of the integers $m_i$, one can find a common integer $m_0$ and positive number $c$ such that $\int\phi_{m_0}\,d\mu <-c$ for every $\mu\in \cI(P)$.

Subadditivity and boundedness above of the $\phi_n$ give, for large $n$ and all $x\in X$,
\begin{equation}\label{eq:block-inequality}
 \phi_n(x)\leq \frac{1}{m_0}\sum_{k=0}^{k_{n}-1}P^k\phi_{m_0}(x)+O(1)
\end{equation}
with $k_{n}=(\lfloor n/m_0\rfloor-1)m_0$, and a uniform remainder in $x$ and $n$ (see \cite[Lem.~2.2.5]{BMbook} for details).
Fix a compact $K$ and $0<\delta<c/m_0$. If the estimate $\sup_{x\in K}\phi_n\leq-\delta n$ failed for arbitrarily large $n$,
there would be $n_j\to \infty$ and $x_j\in K$ with
$\phi_{n_j}(x_j)\geq-\delta n_j$. Set
\[
\eta_j=\frac1{k_{n_j}}\sum_{k=0}^{k_{n_j}-1}(P^k)^*\delta_{x_j}.
\]
Inequality~\eqref{eq:block-inequality} yields
\[
-m_0\delta\leq\frac{k_{n_j}}{n_j}\int\phi_{m_0}\,d\eta_j
+O(1/n_j).
\]
By Proposition~\ref{prop:tightness}, a subsequence of $(\eta_j)_{j\geq1}$
converges to some $\mu\in\cI(P)$. Upper semicontinuity and
boundedness above of $\phi_{m_0}$ then give
$-m_0\delta\leq \int\phi_{m_0}\,d\mu\leq-c$,
a contradiction.
\end{proof}

\subsection{Uniform negativity via a corrector}
Proposition~\ref{Kingman-noncompact} gives uniform estimates on compact
subsets of $X$. We will also need estimates that are uniform on the
whole space. The next proposition provides them through a corrector.

\begin{proposition}\label{prop:corrector}
	Let $(\phi_n)_{n\geq1}$ and $(\psi_n)_{n\geq1}$ be $P$-subadditive sequences. Assume that
	$(\phi_n)_{n\geq1}$ is asymptotically negative, that $\phi_1$ and $\psi_1$ are
	bounded above on $X$, that $(\psi_n)_{n\geq1}$ is uniformly bounded above
	on every compact $K\subset X$,
	and that, for some
	$m\geq1$ and compact $K_0\subset X$,
	\[
	\sup_{x\in X\setminus K_0}\psi_m<0.
	\]
	Then, for every sufficiently large $c>0$, the sequence $(\phi_n^{(c)})_{n\geq1}$, defined by $\phi_n^{(c)}\eqdef\phi_n+c\psi_n$, is asymptotically uniformly negative, in the sense that
	\[
	\limsup_{n\to \infty}\frac1n\sup_{x\in X} \phi_n^{(c)}<0.
	\]
\end{proposition}

We begin with a lemma that is also useful on its own.

\begin{lemma}\label{partition}
Let $(\Phi_n)_{n\geq1}$ be a $P$-subadditive sequence with $\Phi_1$ bounded above. Suppose that $X=X_1\cup\cdots\cup X_r$ and, for some integers $n_i$,
$\sup_{x\in X_i}\frac{1}{n_i}\Phi_{n_i}<\lambda$ for every $i=1,\dots,r$. Then $\sup_{x\in X}\frac1n\Phi_n<\lambda$ for all sufficiently large $n$.
\end{lemma}

\begin{proof}
Choose $\lambda'<\lambda$ with $\sup_{x\in X_i}\Phi_{n_i}/n_i<\lambda'$ and set $\Psi_n=\Phi_n-n\lambda'$, $M_n=\sup_{x\in X}\Psi_n$. Notice that $(\Psi_n)_{n\geq1}$ is $P$-subadditive. For $x\in X_i$ and $n>n_i$ we have $\Psi_{n_i}(x)\leq0$, hence by subadditivity	
	\[
	\Psi_n(x)\leq P^{n_i}\Psi_{n-n_i}(x)+\Psi_{n_i}(x)
	\leq P^{n_i}\Psi_{n-n_i}(x)
	\leq M_{n-n_i}.
	\]
	Put \(N_0=\max_{1\leq i\leq r} n_i\). For \(n>N_0\) and using that $X_1,\dots,X_r$ is a covering of $X$, we obtain
	$M_n\leq \max_{1\leq i\leq r} M_{n-n_i}$. 	By induction, 
	$ M_n\leq M\eqdef \max_{1\leq j\leq N_0}M_j$ for every $n\geq1$. Indeed, this is immediate for \(1\leq n\leq N_0\). If \(n>N_0\) and the claim holds for every \(k<n\), then \(1\leq n-n_i<n\) for every \(i\), and thus
$M_n\leq \max_{1\leq i\leq r}M_{n-n_i}\leq M$,
which completes the induction. 
	Therefore
	$\sup_{x\in X} \frac1n{\Phi_n}\leq \lambda'+\frac{M}{n}$,
	which is less than $\lambda$ for $n$ large enough.
\end{proof}

We now prove Proposition~\ref{prop:corrector}.
\begin{proof}[Proof of Proposition~\ref{prop:corrector}]
Since $\phi_m$ is bounded above, for every sufficiently large $c$, we have 
$\sup_{x\in X\setminus K_0}(\phi_m+c\psi_m)<0$. Fix such a $c>0$.
By asymptotic negativity there is $\delta>0$ such that
$\sup_{x\in K_0}\phi_n\leq-\delta n$ for all large $n$.
The sequence $(\psi_n)_{n\geq1}$ is uniformly bounded above on $K_0$, so we
can choose $n_1$ with $\sup_{x\in K_0}(\phi_{n_1}+c\psi_{n_1})<0$.
Choose $\lambda<0$ larger than both normalized suprema at times $m$
and $n_1$. Lemma~\ref{partition}, applied to
$\Phi_n=\phi_n+c\psi_n$ and the covering
$X=(X\setminus K_0)\cup K_0$, yields the conclusion.
\end{proof}

As an illustration, the drift of a Lyapunov function supplies a
natural corrector.
\begin{corollary}\label{Kingman-noncompact-uniform}
Suppose that $P$ admits a Lyapunov function $V$ and that $(\phi_n)_{n\geq1}$
is a $P$-subadditive sequence with $\phi_1$ bounded above,
which is asymptotically negative. Then, for every sufficiently large $c>0$,
\[
\limsup_{n\to \infty}\frac1n\sup_{x\in X}
\bigl(\phi_n+c(P^n\log V-\log V)\bigr)<0.
\]
\end{corollary}
\begin{proof}
Set $\psi_n=P^n\log V-\log V$. The sequence is $P$-additive.
By Jensen's inequality and Proposition~\ref{LyapunovFunctionsProperties},
$\psi_n\leq\log(P^nV/V)\leq\log C_V$ on $X$, and
$\psi_1\leq\log q_V<0$ outside a compact set.
Proposition~\ref{prop:corrector} applies.
\end{proof}

We conclude with another application of Lemma~\ref{partition}.
\begin{corollary}
	Let $X$ be compact and $(\phi_n)_{n\geq1}$ a $P$-subadditive sequence
	of upper semicontinuous functions. If, for every $x\in X$,
	there is $n\geq1$ such that $\phi_n(x)<0$, then
	$\sup_{x\in X}\phi_n<0$ for all sufficiently large $n$.
\end{corollary}
\begin{proof}
	By upper semicontinuity and compactness, there are a finite
	cover $(B_i)_{1\leq i\leq r}$ of $X$ and integers $(n_i)_{1\leq i\leq r}$ such that
	$\phi_{n_i}<0$ on $B_i$. Lemma~\ref{partition} with
	$\lambda=0$ gives the result.
\end{proof}

\section{Mostly contracting random maps and spectral gap}
\label{sec:mostly}

Let $(T,\mathscr A,p)$ be a probability space,
$(\Omega,\mathscr F,\Pp)=(T^{\N},\mathscr A^{\N},p^{\N})$, and let
$(X,d)$ be locally compact and separable.  We assume $d\leq1$;
any other finite bound can be reduced to this case by rescaling.
A \emph{Bernoulli random map} is a measurable map
$f:\Omega\times X\to X$ of the form $f_\omega=f_{\omega_0}$.
We write $f_\omega^n=f_{\omega_{n-1}}\circ\cdots\circ f_{\omega_0}$.
Its transition probability is
$P(x,B)=\Pp(f_\omega(x)\in B)$, and its annealed Koopman operator is
\[
        P\varphi(x) = \int \varphi(y) \, P(x,dy)
        = \int \varphi(f_\omega(x))\,d\Pp,
\]
which is a Markov--Feller operator. We  call the $P$-stationary probabilities \emph{$f$-stationary measures} and denote by $\cI(f)=\cI(P)$ their set. We assume that \(f_\omega\) is Lipschitz for $\Pp$-a.e.~$\omega$.

For a locally Lipschitz map \(g:X\to X\), define
\[
 L_rg(x)
 \eqdef
 \Lip(g|_{B(x,r)}),
 \qquad
 Lg(x)\eqdef\lim_{r\downarrow 0}L_rg(x).
\]
Then \(Lg\) is upper semicontinuous and 
$L(g\circ h)(x)\leq Lg(h(x))Lh(x)$,~\cite[Lem.~3.1.1]{BMbook}.  Hence the sequence $(\phi_n)_{n\geq1}$ is $P$-subadditive, where 
\[
        \phi_n(x)\eqdef \int\log Lf_\omega^n(x)\,d\Pp.
\]
Under $\int\log^+\Lip(f_\omega)\,d\Pp<\infty$, each $\phi_n$ is upper semicontinuous by reverse Fatou, using the integrable upper bound $\log Lf_\omega^n(x)\leq\sum_{j=0}^{n-1}\log^+\Lip(f_{\omega_j})$. For $\mu\in \cI(f)$, define
\[
 \lambda(\mu)\eqdef \lim_{n\to \infty}\frac1n\int\phi_n\,d\mu
 =\inf_{n\geq1}\frac1n\int\phi_n\,d\mu,
 \qquad
 \lambda(f)\eqdef \sup_{\mu\in\cI(f)}\lambda(\mu).
\]
We call $f$ \emph{mostly contracting} if 
$\lambda(f)<0$.
If, in addition, $P$ admits a Lyapunov function, then under the above
moment assumption Proposition~\ref{Kingman-noncompact} shows that
$(\phi_n)_{n\geq1}$ is asymptotically negative in the sense of
Definition~\ref{assymp-neg}.

Let $V$ be a Lyapunov function. For $\alpha,\beta>0$ set
\[
 \Delta_{\alpha,\beta}(x,y)\eqdef d(x,y)^\alpha V(x)^\beta V(y)^\beta,
 \qquad
 [\varphi]_{\alpha,\beta}\eqdef \sup_{x\neq y}\frac{|\varphi(x)-\varphi(y)|}{\Delta_{\alpha,\beta}(x,y)},
\]
and, fixing $x_0\in X$,
\[
 \Hol_{\alpha,\beta}(X)\eqdef \{\varphi\in C(X;\C):[\varphi]_{\alpha,\beta}<\infty\},
 \qquad
 \|\varphi\|_{\alpha,\beta}\eqdef |\varphi(x_0)|+[\varphi]_{\alpha,\beta}.
\]

\begin{mainthm}
\label{thm:quasicompact}
Let \((T,\mathscr A,p)\) be a probability space, consider
$(\Omega,\mathscr F,\Pp)
        =(T^{\N},\mathscr A^{\N},p^{\N})$,
and let \(X\) be a locally compact separable metric space with distance $d\leq1$. 
Let \(f:\Omega\times X \to X\) be a mostly contracting Lipschitz random map such that for some \(\eta_0>0\),
\begin{equation}\label{eq:random-map-moment}
        \int \Lip(f_\omega)^{\eta_0}\,d\Pp<\infty.
\end{equation}
Assume that its associated annealed Markov operator \(P\) admits a Lyapunov function. Then there is \(\varepsilon_0>0\) such that, for any $\alpha,\beta>0$ with $\beta<\varepsilon_0$ and $\alpha/\beta<\varepsilon_0$, 
\[
        P:\Hol_{\alpha,\beta}(X)\longrightarrow
        \Hol_{\alpha,\beta}(X)
\]
is quasi-compact.
\end{mainthm}

Global contraction on average and a spectral gap follow under an
additional assumption. For a map $g:X\to X$, its two-point motion is
$g^{[2]}:X\times X\to X\times X$, defined by
$$g^{[2]}(x,y)=(g(x),g(y)).$$ Let $f^{[2]}$ be the corresponding two-point motion random
map and $P^{[2]}$ its Markov operator. We call $f$
\emph{stationarily proximal} if every $P^{[2]}$-stationary probability
is concentrated on the diagonal $\mathcal D=\{(x,x):x\in X\}$.

\begin{mainthm}
\label{thm:product-weighted-global-contraction}
Under the assumptions of Theorem~\ref{thm:quasicompact}, assume additionally that $f$ is stationarily proximal. Then there exists $\varepsilon_0>0$ such that for any $\alpha, \beta > 0$ with $0<\beta<\varepsilon_0$ and $\alpha/\beta<\varepsilon_0$, there are $C>0$ and $\rho \in (0,1)$ satisfying
\begin{equation}
\label{eq:product-global-contraction}
 \int
 \Delta_{\alpha,\beta}
 \bigl(f_\omega^n(x),f_\omega^n(y)\bigr)\,d\Pp
 \leq
 C\rho^n\Delta_{\alpha,\beta}(x,y) \quad \text{for every \(x,y\in X\) and \(n\geq0\).}
\end{equation}
 Moreover, $P:\Hol_{\alpha,\beta}(X)\to\Hol_{\alpha,\beta}(X)$
has a spectral gap.  More precisely, \(P\) has a unique stationary probability measure \(\nu\),
and,  after changing \(C\) if
necessary,
\begin{equation}
\label{eq:product-spectral-gap}
 \left\|P^n\varphi-\int \varphi \,d\nu \right\|_{\alpha,\beta}
 \leq
 C\rho^n\bigl\|\varphi\bigr\|_{\alpha,\beta} \quad \text{
for every \(\varphi\in\Hol_{\alpha,\beta}(X)\) and \(n\geq0\)}.
\end{equation}
\end{mainthm}

As a consequence of the previous theorem, we get the following result.

\begin{corollary}\label{prop:synchronization-equivalences}
Under the hypotheses of Theorem~\ref{thm:quasicompact}, the following are equivalent:
\begin{enumerate}[itemsep=0.1cm]
\item[\rm(i)] $f$ is stationarily proximal;
\item[\rm(ii)] $f$ is \emph{globally contracting on average}: for some $\alpha,\beta>0$, $C>0$, and $\rho\in(0,1)$,
\begin{equation}\label{eq:equivalent-global-contraction}
 \qquad \int\Delta_{\alpha,\beta}(f_\omega^n(x),f_\omega^n(y))\,d\Pp
 \leq C\rho^n\Delta_{\alpha,\beta}(x,y)
 \quad \text{for every \(x,y\in X\) and \(n\geq0\);}
\end{equation}
\item[\rm(iii)] $f$ is \emph{synchronizing}: for every $x,y\in X$, $d(f_\omega^n(x),f_\omega^n(y))\to 0$ for $\Pp$-a.e.~$\omega$;
\item[\rm(iv)] $f$ is \emph{synchronizing in mean}: for every $x,y\in X$, $\int d(f_\omega^n(x),f_\omega^n(y))\,d\Pp\to 0$;
\item[\rm(v)] $f^{[2]}$ is uniquely ergodic (i.e., $P^{[2]}$ has a unique stationary probability).
\item[\rm(vi)] \(f\) is \emph{proximal}: for every \(x,y\in X\),
there are a compact set \(K\subset X\) and a sequence
\((g_n)_{n\geq1}\) in the semigroup generated by the support of the law
of the fiber maps such~that
$$ g_n(x),g_n(y)\in K \quad  \text{and} \quad  d(g_n(x),g_n(y))\to 0.$$
\end{enumerate}
When these six properties hold, \eqref{eq:equivalent-global-contraction} holds for every sufficiently small $\alpha,\beta>0$ with $\alpha/\beta$ sufficiently small, and $P$ has a spectral gap on $\Hol_{\alpha,\beta}(X)$.
\end{corollary}


\begin{proof} ~\\[-0.6cm]
\begin{enumerate}[label=$\bullet$, leftmargin=0.5cm]
\item \textup{(i)$\Rightarrow$(ii).} This is Theorem~\ref{thm:product-weighted-global-contraction}.

\item \textup{(ii)$\Rightarrow$(iii).} For fixed $x,y\in X$, condition (ii) gives
$\int\sum_{n\geq0}\Delta_{\alpha,\beta}(f_\omega^n(x),f_\omega^n(y))\,d\Pp<\infty$.
Thus the summands tend to zero almost surely. Since
$\Delta_{\alpha,\beta}(u,v)\geq d(u,v)^\alpha$, this gives~(iii).

\item \textup{(iii)$\Rightarrow$(iv).} This follows from dominated convergence,
since $d\leq1$.

\item \textup{(iv)$\Rightarrow$(i).} Let $\pi$ be any $P^{[2]}$-stationary
probability. For every $n\geq1$, stationarity gives
\[
 \int d(x,y)\,d\pi
 =\int\!\!\int d(f_\omega^n(x),f_\omega^n(y))\,d\Pp\,d\pi.
\]
The right-hand side tends to zero by (iv) and dominated convergence.
Hence $\pi$ is supported on the diagonal.

\item \textup{(i)$\Rightarrow$(v).} Let $\mu$ and $\mu'$ be $P$-stationary
probabilities. The Ces\`aro averages
\[
 \pi_N=\frac1N\sum_{k=0}^{N-1}(P^{[2]})^{*k}(\mu\times\mu')
\]
all have marginals $\mu$ and $\mu'$, hence are tight. Every weak$^*$
limit point is $P^{[2]}$-stationary, because $P^{[2]}$ is Feller.
By (i) it is diagonal, so $\mu=\mu'$. Thus $P$ has a unique
stationary probability $\nu$. Every $P^{[2]}$-stationary probability
is diagonal by (i) and has marginal $\nu$; it must therefore equal
$(x\mapsto(x,x))_*\nu$.

\item \textup{(v)$\Rightarrow$(i).} Proposition~\ref{prop:tightness}
provides a $P$-stationary probability $\nu$. Its diagonal image
$(x\mapsto(x,x))_*\nu$ is $P^{[2]}$-stationary and, by (v), is the
only such probability. 
\item {\rm(iii)}\(\Rightarrow\){\rm(vi)}.
The function \(S(x,y)\eqdef V(x)+V(y)\) is a Lyapunov function
for \(P^{[2]}\). Hence, by Proposition~\ref{prop:tightness},
for every \(x,y\in X\), the empirical measures of the Markov chain 
\(((f^{[2]})_\omega^n(x,y))_{n\geq0}\)
are tight for \(\Pp\)-a.e.~\(\omega\).
Choose \(\omega\) for which this property and {\rm(iii)} hold.
Tightness gives a compact set \(L\subset X\times X\) and \(n_j\to\infty\) such that $        (f^{[2]})_\omega^{n_j}(x,y)\in L$
        for $j\geq1$. 
Set \(K\eqdef\mathrm{proj}_1(L)\cup\mathrm{proj}_2(L)\).
Then \(K\) is compact,
\(f_\omega^{n_j}(x),f_\omega^{n_j}(y)\in K\) and, by~{\rm(iii)},
$d(f_\omega^{n_j}(x),f_\omega^{n_j}(y))
        \to 0$. Since \(f_\omega^{n_j}\in G\), where \(G\) is the semigroup
in {\rm(vi)}, this proves {\rm(vi)}.

\item {\rm(vi)}\(\Rightarrow\){\rm(v)}.
By Proposition~\ref{prop:tightness}, \(f\) is tight in the sense
of \cite[Thm.~10.2.1]{BMbook}. Since \(f\) is mostly contracting,
the argument of that theorem gives the \emph{local contraction
property}: for every \(x\in X\) and \(\Pp\)-a.e.~\(\omega\),
there is a neighborhood \(B\) of \(x\) such that $    \mathrm{diam}\,f_\omega^n(B)\to 0$.
As observed in \cite[proof of Prop.~3.5.1]{BMbook},
\(f^{[2]}\) also has the local contraction property.
Suppose that \(\pi_1,\pi_2\) are distinct ergodic
\(P^{[2]}\)-stationary probabilities and put
\(F_i\eqdef\supp\pi_i\).
Then \(F_1\cap F_2=\emptyset\).
Indeed, if \(u\in F_1\cap F_2\), local contraction gives a
contractible ball \(B\) containing \(u\), and
\(\pi_1(B)\pi_2(B)>0\).
However, the argument of \cite[Lem.~4.7]{Mal17}, using bounded
Lipschitz test functions, shows that a contractible ball can
have positive mass for at most one ergodic stationary measure.
Thus \(\pi_1=\pi_2\), a contradiction.

Now choose \((x_i,y_i)\in F_i\), \(i=1,2\).
For each \(i\), {\rm(vi)} gives a compact set \(K_i\subset X\)
and a sequence \((g_{i,n})_{n\geq1}\subset G\) such that
$ g_{i,n}(x_i),g_{i,n}(y_i)\in K_i$, $        d(g_{i,n}(x_i),g_{i,n}(y_i))\to 0$. Passing to a subsequence gives \(z_i\in K_i\) such that
\(g_{i,n}(x_i)\to z_i\) and \(g_{i,n}(y_i)\to z_i\).
Since \(F_i\) is closed and invariant under the diagonal
action of \(G\), \((z_i,z_i)\in F_i\).
Applying {\rm(vi)} once more to \(z_1,z_2\), and using the
compact set in the definition of proximality, we obtain a
point \(z\in X\) such that $        (z,z)\in F_1\cap F_2$, a contradiction.
Thus \(P^{[2]}\) has at most one ergodic stationary probability.
Since \(S\) is a Lyapunov function for \(P^{[2]}\),
Proposition~\ref{prop:tightness} gives existence of stationary
probabilities. By ergodic decomposition, \(f^{[2]}\) is
therefore uniquely ergodic.
This proves {\rm(vi)}\(\Rightarrow\){\rm(v)}. \qedhere
\end{enumerate}
\end{proof}

\subsection{Contraction properties}

We establish the estimates on $\Delta_{\alpha,\beta}$ needed in the proofs
of Theorems~\ref{thm:quasicompact}
and~\ref{thm:product-weighted-global-contraction}.

\begin{proposition}\label{prop:contractions}
Under the assumptions of Theorem~\ref{thm:quasicompact}, there is
$\varepsilon_0>0$ such that, whenever $0<\beta<\varepsilon_0$ and
$0<\alpha/\beta<\varepsilon_0$, there are $C>0$ and $\rho\in(0,1)$
with the following properties:
\begin{enumerate}[label=(\roman*)]
\item \emph{Local contraction.} For every $n\geq0$ and $x\in X$,
\begin{equation}\label{eq:local-weighted-contraction}
\limsup_{\substack{y,z\to x\\y\neq z}}
\int\frac{\Delta_{\alpha,\beta}(f_\omega^n(y),f_\omega^n(z))}
{\Delta_{\alpha,\beta}(y,z)}\,d\Pp\leq C\rho^n.
\end{equation}
\item \emph{Global contraction under stationary proximality.} If $f$ is
stationarily proximal, then, for every $n\geq0$,
\begin{equation}\label{eq:global-weighted-contraction}
\sup_{x\neq y}
\int\frac{\Delta_{\alpha,\beta}(f_\omega^n(x),f_\omega^n(y))}
{\Delta_{\alpha,\beta}(x,y)}\,d\Pp\leq C\rho^n.
\end{equation}
\item \emph{Contraction at infinity.} For each $n\geq0$ there is a compact
set $K_n\subset X$ such that
\begin{equation}\label{eq:infinity-weighted-contraction}
	\sup_{\substack{(x,y)\notin K_n\times K_n\\x\neq y}}
	\int\frac{\Delta_{\alpha,\beta}(f_\omega^n(x),f_\omega^n(y))}
	{\Delta_{\alpha,\beta}(x,y)}\,d\Pp\leq C\rho^n.
\end{equation}
\end{enumerate}
\end{proposition}

\subsubsection{The cocycles behind the contraction estimates}
The weighted quotients in Proposition~\ref{prop:contractions}
suggest two versions of the same probabilistic estimate. For the
local estimate, continuity of $V$ gives
\[
\limsup_{\substack{y,z\to x\\y\neq z}}
\frac{\Delta_{\alpha,\beta}(f_\omega^n(y),f_\omega^n(z))}
     {\Delta_{\alpha,\beta}(y,z)}
=A_n(\omega,x)^\alpha B_n(\omega,x)^\beta,
\]
where
\begin{equation} \label{eq:localpair}
A_n(\omega,x)=Lf_\omega^n(x),\qquad
B_n(\omega,x)
=\left(\frac{V(f_\omega^n(x))}{V(x)}\right)^2.
\end{equation}
For the global estimate, define for every Lipschitz map $g:X\to X$
\[
\widehat Lg(x,y)=
\begin{cases}
d(g(x),g(y))/d(x,y),&x\neq y,\\
Lg(x),&x=y.
\end{cases}
\]
On $X^2$ we use
\[
A_n^{[2]}(\omega;x,y)=\widehat Lf_\omega^n(x,y),\qquad
B_n^{[2]}(\omega;x,y)
=\frac{V(f_\omega^n(x))V(f_\omega^n(y))}{V(x)V(y)}.
\]
Then, for $x\neq y$,
\[
\frac{\Delta_{\alpha,\beta}(f_\omega^n(x),f_\omega^n(y))}
     {\Delta_{\alpha,\beta}(x,y)}
=A_n^{[2]}(\omega;x,y)^\alpha B_n^{[2]}(\omega;x,y)^\beta.
\]
The two pairs agree on the diagonal, but we will apply the general
proposition below on $X$ for local contraction and on $X^2$ for global contraction.
The drift of $V$ controls the moments of $B_n$ and its decay at
infinity; negativity of the averaged logarithms of $A_n$ supplies
the other ingredient. The latter follows from mostly contracting
on $X$, and on $X^2$ when stationary proximality is assumed.

\subsubsection{A general cocycle estimate}
We isolate the estimate suggested by the two weighted
quotients above. In the next proposition, $f$ may be any Bernoulli random map
on a locally compact space $X$, with Markov operator $P$: no Lyapunov
function or mostly contracting hypothesis is imposed. Here $A_n,B_n$
denote arbitrary cocycles, of which the two pairs above are examples.
We call \emph{submultiplicative cocycle} over $f$ a family
$A_n:\Omega\times X\to[0,\infty)$, $n\geq1$, with
$A_n(\cdot,x)$ depending only on $\omega_0,\ldots,\omega_{n-1}$,
such that
\begin{equation}\label{eq:submultiplicative-cocycle}
A_{n+m}(\omega,x)\leq
A_m(\omega,x)\,
A_n(\sigma^m\omega,f_\omega^m(x)),
\end{equation}
where $\sigma$ is the left shift on $\Omega$. 

Let $(A_n)_{n\geq1}$ and $(B_n)_{n\geq1}$ be two such cocycles. Suppose that for some
$r>0$, $L,M\geq1$, $q\in(0,1)$, and compact sets $K_n\subset X$,
\begin{equation}\label{eq:two-cocycle-assumptions}
\sup_{x\in X}\int A_1(\omega,x)^r\,d\Pp\leq L,\qquad
\sup_{n\geq1,\,x\in X}\int B_n(\omega,x)^r\,d\Pp\leq M,
\end{equation}
\begin{equation}\label{eq:two-cocycle-drift}
\sup_{x\in X\setminus K_n}
\int B_n(\omega,x)^r\,d\Pp\leq Mq^n
\qquad \text{for } n\geq1.
\end{equation}

\begin{proposition}\label{prop:two-cocycle-estimates}
There is $\varepsilon>0$ such that for every $0<\beta<\varepsilon$
and $0<\alpha/\beta<\varepsilon$, some $C>0$ and $\rho\in(0,1)$
satisfy
\begin{equation}\label{eq:two-cocycle-infinity}
\sup_{x\in X\setminus K_n}
\int A_n(\omega,x)^\alpha B_n(\omega,x)^\beta\,d\Pp
\leq C\rho^n\qquad \text{for } n\geq1.
\end{equation}
If, in addition, $A_n(\omega,\cdot)$ and $B_n(\omega,\cdot)$ are
upper semicontinuous for almost every~$\omega$, for every $n$, and
the sequence $\bigl(x\mapsto\int\log A_n(\omega,x)\,d\Pp\bigr)_{n\geq1}$
is asymptotically negative in the sense of Definition~\ref{assymp-neg},
then, after reducing $\varepsilon$ if necessary, the same conclusion
holds uniformly on $X$:
\begin{equation}\label{eq:two-cocycle-contraction}
\sup_{x\in X}
\int A_n(\omega,x)^\alpha B_n(\omega,x)^\beta\,d\Pp
\leq C\rho^n\qquad\text{for } n\geq1.
\end{equation}
\end{proposition}
\begin{proof}
Replacing $A_n,B_n$ by $A_n^r,B_n^r$ and rescaling the exponents,
we may assume $r=1$. Submultiplicativity and independence give
\begin{equation}\label{eq:two-cocycle-A-moment}
\sup_{x\in X}\int A_n(\omega,x)\,d\Pp\leq L^n.
\end{equation}
When $\alpha+\beta\leq1$, H\"older's inequality yields, outside $K_n$,
\[
\int A_n(\omega,x)^\alpha B_n(\omega,x)^\beta\,d\Pp
\leq L^{\alpha n}(Mq^n)^\beta
=M^\beta(L^\alpha q^\beta)^n.
\]
Taking $\beta$ and $\alpha/\beta$ sufficiently small proves
\eqref{eq:two-cocycle-infinity}.

Now assume the additional hypotheses. Set
$\phi_n(x)=\int\log A_n(\omega,x)\,d\Pp$ and
$\psi_n(x)=\int\log B_n(\omega,x)\,d\Pp$.
These are $P$-subadditive sequences, and Jensen's inequality gives
$\phi_n\leq n\log L$ and $\psi_n\leq\log M$ on $X$. Choose $n_0$
with $Mq^{n_0}<1$. Then $\psi_{n_0}<0$ on $X\setminus K_{n_0}$,
and $(\psi_n)_{n\geq1}$ is uniformly bounded above on $X$. Thus
Proposition~\ref{prop:corrector} applies to
$\phi_n,\psi_n$. Choose \(c>0\) sufficiently large that the conclusion of
Proposition~\ref{prop:corrector} holds and \(1/c<\varepsilon\),
where \(\varepsilon\) is the constant in the first assertion
of the present proposition.
With $C_n\eqdef A_nB_n^c$, we obtain
\begin{equation}\label{eq:two-cocycle-negative-log}
\limsup_{n\to \infty}\frac1n\sup_{x\in X}
\int\log C_n(\omega,x)\,d\Pp<0.
\end{equation}

Choose \(0<s_\infty<\varepsilon/c\).
The first assertion applies with \(\alpha=s_\infty\) and
\(\beta=cs_\infty\), since
\(\beta<\varepsilon\) and \(\alpha/\beta=1/c<\varepsilon\). This gives
$\sup_{x\in X\setminus K_n}\int C_n^{s_\infty}\,d\Pp\leq C_\infty\rho_\infty^n$.
Choose $N$ large enough that the latter bound is less than $1$ and
\eqref{eq:two-cocycle-negative-log} holds for $n=N$; put $K=K_N$.
For a sufficiently small $s_0>0$, the moment assumptions give
$\sup_{x\in X}\int C_N^{2s_0}\,d\Pp<\infty$.
The pointwise upper semicontinuity and
Corollary~\ref{cor:semicontinuity} therefore show that
$x\mapsto\int C_N(\omega,x)^s\,d\Pp$ is upper semicontinuous
for $0<s\leq s_0$.
By Lemma~\ref{lem:one-sided-small-powers} and
\eqref{eq:two-cocycle-negative-log}, for each $x\in K$ there is
$s_x\in(0,s_0]$ such that $\int C_N(\omega,x)^{s_x}\,d\Pp<1$.
A finite cover of $K$ and Jensen's inequality allow the exponents
to be decreased to a common $s\leq s_\infty$ for which this
inequality holds on $K$. Jensen also preserves the strict bound on
$X\setminus K$. Thus
\begin{equation}\label{eq:two-cocycle-block}
\sup_{x\in X}\int C_N(\omega,x)^s\,d\Pp<1.
\end{equation}

Put $a_n\eqdef \sup_{x\in X}\int C_n^s\,d\Pp$. These numbers are finite by the
moment assumptions, and submultiplicativity and independence give
$a_{n+m}\leq a_na_m$. Hence $a_N<1$ implies
$a_n\leq C_0\rho_0^n$ for some $C_0>0$ and $\rho_0<1$.
Finally choose $\beta$ and $\alpha/\beta$ small enough that
$t\eqdef \beta-c\alpha\geq0$ and $\alpha/s+t\leq1$. H\"older's inequality
gives
\[
\int A_n(\omega,x)^\alpha B_n(\omega,x)^\beta\,d\Pp
=\int C_n(\omega,x)^\alpha B_n(\omega,x)^t\,d\Pp
\leq a_n^{\alpha/s}
M^t
\leq C_0^{\alpha/s}M^t\rho_0^{\alpha n/s},
\]
uniformly in $x$, as required.
\end{proof}

\subsubsection{Proof of Proposition~\ref{prop:contractions}} ~\\[-0.25cm]

 \emph{(i) Local contraction.}
 Apply the uniform conclusion of
Proposition~\ref{prop:two-cocycle-estimates} on $X$ to the local
pair $A_n,B_n$ defined in~\eqref{eq:localpair}.
\begin{itemize}
	\item The local Lipschitz constants form a submultiplicative cocycle
	by the chain rule and are upper semicontinuous in $x$; $B_n$ is a continuous
	multiplicative cocycle.
	\item For $r\leq\eta_0$,
	$\int A_1(\omega,x)^r\,d\Pp
	\leq\int\Lip(f_\omega)^r\,d\Pp<\infty$.
	\item For $r\leq1/2$, $\int B_n(\omega,x)^r\,d\Pp \leq \left(\frac{P^nV(x)}{V(x)}\right)^{2r}$, with $P^nV/V$ bounded by $C_V$ on $X$ and by $q_V^n$ outside a compact $K_n$.
	\item The averaged logarithms $\int\log A_n(\omega,\cdot)\,d\Pp$
	form the sequence $(\phi_n)_{n\geq1}$, which is asymptotically negative by
	Proposition~\ref{Kingman-noncompact}.
\end{itemize}
The proposition gives: if $\beta$ and $\alpha/\beta$ are small enough, we obtain for some $C>0$ and $\rho<1$
\begin{equation}\label{eq:local-cocycle-estimate}
\int
\limsup_{\substack{y,z\to x\\y\neq z}}
\frac{\Delta_{\alpha,\beta}(f_\omega^n(y),f_\omega^n(z))}
{\Delta_{\alpha,\beta}(y,z)}
\,d\Pp\leq C\rho^n.
\end{equation}
We conclude by applying Lemma~\ref{FatouReverse} to any sequence of distinct pairs $(y,z)\to(x,x)$. Its moment condition holds for sufficiently small $\alpha,\beta$: for each fixed $n$,
$$\int\left(
\frac{\Delta_{\alpha,\beta}(f_\omega^n(y),f_\omega^n(z))}
{\Delta_{\alpha,\beta}(y,z)}\right)^2\,d\Pp$$ is bounded uniformly in $y\neq z$ by H\"older, the moment condition on $\Lip(f_\omega)$ and boundedness of $P^nV/V$ (for instance, $6\alpha\leq\eta_0$ and $6\beta\leq1$ suffice). This proves (i).

\emph{(ii) Global contraction.}
The function $\widehat Lg$ is upper semicontinuous on $X^2$, bounded by
$\Lip(g)$, and satisfies the chain-rule inequality. Apply the uniform conclusion of
Proposition~\ref{prop:two-cocycle-estimates} to the two-point map on
$X^2$, with $A_n=A_n^{[2]}$ and $B_n=B_n^{[2]}$.
Its operator $P^{[2]}$ is Feller, and
$W(x,y)=V(x)+V(y)$ is a Lyapunov function.
These cocycles have the required pointwise semicontinuity, and
the moment of $A_1$ follows from
$\widehat Lf_\omega\leq\Lip(f_\omega)$.
For $2r\leq1$, H\"older and Jensen give
\[
\int B_n(\omega;x,y)^r\,d\Pp
\leq\left(
\frac{P^nV(x)}{V(x)}
\frac{P^nV(y)}{V(y)}
\right)^r.
\]
This is bounded by $C_V^{2r}$ on $X^2$ and by
$C_V^rq_V^{rn}$ outside $K_n\times K_n$.
The averaged logarithms of $A_n^{[2]}$ are upper
semicontinuous, by Corollary~\ref{cor:semicontinuity} and the moment
bound. Every $P^{[2]}$-stationary probability is supported on the
diagonal by stationary proximality. There
$A_n^{[2]}(\omega;z,z)=Lf_\omega^n(z)$, so these averaged logarithms
satisfy the stationary criterion for $P^{[2]}$, since $\lambda(f)<0$.
Proposition~\ref{Kingman-noncompact}, applied on $X^2$ with
the Lyapunov function $W$, gives asymptotic negativity. The
general proposition now gives (ii).

\emph{(iii) Contraction at infinity.} Apply the first conclusion of
Proposition~\ref{prop:two-cocycle-estimates} on $X^2$ with
$A_n=A_n^{[2]}$ and $B_n=B_n^{[2]}$. The required moment bounds
were verified above. This conclusion does not use stationary
proximality. It gives (iii) for $n\geq1$; enlarge $C$ to include $n=0$.
Take a common bound on $\beta$ and $\alpha/\beta$ for (i)--(iii).

\subsection{Weighted spaces}

For \(\gamma>0\), put
\[
 |\varphi|_\gamma
 \eqdef
 \sup_{x\in X}\frac{|\varphi(x)|}{V(x)^\gamma},
 \qquad
 C_\gamma(X)
 \eqdef
 \{\varphi\in C(X;\C):|\varphi|_\gamma<\infty\}.
\]

\begin{lemma}\label{lem:compact-embedding}
The spaces \((\Hol_{\alpha,\beta}(X),\|\cdot\|_{\alpha,\beta})\) and \((C_\gamma(X),|\cdot|_\gamma)\) are Banach.
If \(\gamma\geq\beta\), then
\(\Hol_{\alpha,\beta}(X)\hookrightarrow C_\gamma(X)\) continuously.
If \(\gamma>\beta\), the embedding is compact.
\end{lemma}

\begin{proof}
Since the map $\varphi\longmapsto V^{-\gamma}\varphi$
is an isometric isomorphism from \(C_\gamma(X)\) onto \(C_b(X;\C)\), the space \(C_\gamma(X)\) is Banach.
Completeness of $\Hol_{\alpha,\beta}(X)$ follows by the standard argument for weighted H\"older
spaces; compare~\cite[Sec.~5]{walkden2013transfer}.  On the other hand, for \(\gamma\geq\beta\), the norm $|\varphi|_\gamma+[\varphi]_{\alpha,\beta}$ 
is equivalent to \(\|\varphi\|_{\alpha,\beta}\).  Indeed,
\[
 |\varphi(x_0)|
 \leq V(x_0)^\gamma|\varphi|_\gamma,
 \qquad
 |\varphi|_\gamma
 \leq
 |\varphi(x_0)|
 +V(x_0)^\beta[\varphi]_{\alpha,\beta}.
\]
In particular, there is a constant \(C>0\) such that
$ |\varphi|_\gamma
    \leq C\|\varphi\|_{\alpha,\beta}$,
which proves the continuity of the embedding
\(\Hol_{\alpha,\beta}(X)\hookrightarrow C_\gamma(X)\).

If \(\gamma>\beta\), let $(\varphi_j)_{j\geq1}$ be bounded in
$\Hol_{\alpha,\beta}(X)$. On every compact sublevel of $V$, it is
uniformly bounded and equicontinuous. Hence Arzel\`a--Ascoli and a
diagonal argument give a locally uniformly convergent subsequence.
Moreover,
$\sup_{V(x)>R}|\varphi_j(x)|V(x)^{-\gamma}
 \leq C R^{\beta-\gamma}$,
uniformly in $j$. Since $\gamma>\beta$, the local uniform convergence
therefore upgrades to convergence in $C_\gamma(X)$, proving compactness.
\end{proof}

\begin{lemma}
	The operator \(P\) is bounded on
	\(\Hol_{\alpha,\beta}(X)\) and \(C_\gamma(X)\)
	for all sufficiently small \(\alpha,\beta,\gamma>0\).
\end{lemma}
\begin{proof}
For $\varphi\in C_\gamma(X)$ and $0<\gamma<1$, $P\varphi$ is continuous
by uniform integrability: if $x_j\to x$, then $f_\omega(x_j)\to f_\omega(x)$
almost surely and
$\sup_{j\geq1}\int|\varphi(f_\omega(x_j))|^{1/\gamma}\,d\Pp<\infty$,
since $\sup_{j\geq1} PV(x_j)<\infty$. Jensen's inequality gives
	${|P\varphi(x)|}/{V(x)^\gamma}
	\leq
	|\varphi|_\gamma
	(PV(x)/V(x))^\gamma$,
	and hence $|P\varphi|_\gamma\leq C|\varphi|_\gamma$ for some $C$.
	
	For $\Hol_{\alpha,\beta}(X)$, if $\alpha$ and $\beta$ are sufficiently
	small, then $$A\eqdef \sup_{x\neq y}\int
	\frac{\Delta_{\alpha,\beta}(f_\omega(x),f_\omega(y))}
	{\Delta_{\alpha,\beta}(x,y)}
	\,d\Pp<\infty.$$ This follows from H\"older's inequality, the
	moment assumption on $\Lip(f_\omega)$, and the bound on $PV/V$.
	Consequently, $[P\varphi]_{\alpha,\beta}
	\leq A[\varphi]_{\alpha,\beta}$. It remains to control
	$|P\varphi(x_0)|$, using
	\[
	|P\varphi(x_0)| \leq 
	|\varphi(x_0)|
	+[\varphi]_{\alpha,\beta}
	\int\Delta_{\alpha,\beta}(f_\omega(x_0),x_0)\,d\Pp,\]
	where $\int\Delta_{\alpha,\beta}(f_\omega(x_0),x_0)\,d\Pp \leq  V(x_0)^\beta\bigl(PV(x_0)\bigr)^\beta<+\infty$.
	Thus, 
	\[\|P\varphi\|_{\alpha,\beta}
	=|P\varphi(x_0)|+[P\varphi]_{\alpha,\beta}
	\leq C\|\varphi\|_{\alpha,\beta}.\]
	Hence $P$ is bounded on $\Hol_{\alpha,\beta}(X)$.
\end{proof}

\subsection{Quasi-compactness}
We write $r(\cdot)$ and $r_{\mathrm{ess}}(\cdot)$ for the spectral
and essential spectral radii, respectively. Here we prove Theorem~\ref{thm:quasicompact} using  Proposition~\ref{prop:contractions}(i) and (iii).

\begin{lemma}\label{cor:one-block}
	For \(\alpha,\beta\) as in Proposition~\ref{prop:contractions}, there are
	\(N\geq1\), a compact set \(K\subset X\), and \(r>0\) such that
	\[
	(x,y)\notin K\times K
	\quad\text{or}\quad
	d(x,y)\leq r
	\quad\Longrightarrow\quad
	\int \Delta_{\alpha,\beta}(f_\omega^N(x),f_\omega^N(y)) \,d\Pp
	\leq\frac12\Delta_{\alpha,\beta}(x,y).
	\]
\end{lemma}

\begin{proof}
	We apply Proposition~\ref{prop:contractions}, choosing $N$ such that $C\rho^N<1/2$. If $K=K_N$ is the compact set given by the proposition, then, by item~(iii), the conclusion holds when $(x,y)\notin K\times K$. We now consider $x,y\in K$. Proposition~\ref{prop:contractions}\textup{(i)} gives for each $z\in X$ that
	\[
	\limsup_{\substack{x,y\to z\\x\neq y}}
	\int \frac{\Delta_{\alpha,\beta}(f_\omega^N(x),f_\omega^N(y))}
	{\Delta_{\alpha,\beta}(x,y)}\,d\Pp <\frac12.
	\]
	Hence every \(z\in K\) has an open neighborhood \(U_z\) such that
	\begin{equation} \label{eq:distinct}
		\int \frac{\Delta_{\alpha,\beta}(f_\omega^N(x),f_\omega^N(y))}
		{\Delta_{\alpha,\beta}(x,y)} \,d\Pp<\frac12
		\qquad
		\text{for all distinct }x,y\in U_z.
	\end{equation}
	Choose a finite subcover $K\subset U_{z_1}\cup\cdots\cup U_{z_m}$
	and let \(\ell>0\) be a Lebesgue number for the induced cover of \(K\). Taking \(0<r<\ell\), any  \(x,y\in K\) with \(0<d(x,y)\leq r\) belong to a common \(U_{z_i}\), and hence~\eqref{eq:distinct} holds, completing the proof.
\end{proof}

Thanks to this lemma, we can prove a Lasota-Yorke inequality.

\begin{proposition}\label{prop:LY}
Choose \(\alpha,\beta\) as in Lemma~\ref{cor:one-block} and
\(\beta<\gamma<1\).
There are \(N\geq1\), \(\rho\in(0,1)\), and
\(C>0\) such that
\begin{equation}\label{eq:LY}
        [P^N\varphi]_{\alpha,\beta}
        \leq
        \rho[\varphi]_{\alpha,\beta}
        +C|\varphi|_\gamma \quad \text{for every \(\varphi\in\Hol_{\alpha,\beta}(X)\).}
\end{equation}
\end{proposition}

\begin{proof}
Let \(K,r,N\) be supplied by Lemma~\ref{cor:one-block}.  If
\((x,y)\notin K\times K\) or \(d(x,y)\leq r\), then
\begin{equation}\label{eq:bound-Cc}
 |P^N\varphi(x)-P^N\varphi(y)| \leq \int | \varphi(f_\omega^N(x))-\varphi(f_\omega^N(y)) | \,d\Pp 
 \leq
 \frac12[\varphi]_{\alpha,\beta}\Delta_{\alpha,\beta}(x,y).
\end{equation}
The remaining pairs belong to the compact set $  \mathcal C
        =\{(x,y)\in K\times K:d(x,y)\geq r\}$.
On \(\mathcal C\), \(\Delta_{\alpha,\beta}\) has a positive minimum $\delta_{\mathcal{C}} \eqdef  \min_{(x,y)\in\mathcal{C}} \Delta_{\alpha,\beta}(x,y) > 0$, while
Jensen's inequality and the boundedness of $P^NV$  on $K$ give for $x\in K$ that 
\[
 |P^N\varphi(x)|
 \leq \int |\varphi(f_\omega^N(x)) |\,d\Pp \leq 
 |\varphi|_\gamma \int V(f_\omega^N(x))^\gamma \,d\Pp
 \leq
 |\varphi|_\gamma(P^NV(x))^\gamma
 \leq C_K |\varphi|_\gamma.
\]
Applying the triangle inequality gives
\begin{equation}\label{eq:bound-C}
\sup_{(x,y) \in \mathcal{C}} \frac{|P^N\varphi(x)-P^N\varphi(y)|}{\Delta_{\alpha,\beta}(x,y)} \leq C |\varphi|_\gamma \quad \text{where} \ \ C \eqdef  \frac{2C_K}{\delta_{\mathcal{C}}}
\end{equation}
Taking the global supremum over all $x \neq y$, equations~\eqref{eq:bound-Cc} and~\eqref{eq:bound-C} combine to show that for any choice of $\rho \in (1/2, 1)$, 
\[
[P^N\varphi]_{\alpha,\beta} \leq \max\left\{ \frac{1}{2}[\varphi]_{\alpha,\beta}, \, C|\varphi|_\gamma \right\} \leq \rho[\varphi]_{\alpha,\beta} + C|\varphi|_\gamma,
\]
which proves the Lasota--Yorke inequality~\eqref{eq:LY}. 
\end{proof}

\begin{proof}[Proof of Theorem~\ref{thm:quasicompact}] Apply Proposition~\ref{prop:LY}
with \(\beta<\gamma<1\).
Lemma~\ref{lem:compact-embedding} makes the strong unit ball relatively
compact in the weak norm. Moreover, the drift estimate and Jensen's inequality give
\( |P^n\varphi|_\gamma\leq C_V^\gamma|\varphi|_\gamma \) for every \(n\geq0\), so \(P^n\) is uniformly bounded in the weak norm.
The Ionescu--Tulcea--Marinescu--Hennion
quasi-compactness criterion applied to~\eqref{eq:LY} yields
$r_{\mathrm{ess}}\big(P|\Hol_{\alpha,\beta}(X)\big)
        \leq\rho^{1/N}<1$.
\end{proof}
\subsection{Spectral gap: proof of  Theorem~\ref{thm:product-weighted-global-contraction}}
Proposition~\ref{prop:contractions} gives the contraction inequality;
it remains to prove the spectral gap. Decreasing $\beta$ if necessary,
assume $\beta<1$. For all $x,y\in X$,
\begin{align*}
 |P^n\varphi(x)-P^n\varphi(y)|
 &\leq
 \int
 |\varphi(f_\omega^n(x))-\varphi(f_\omega^n(y))|
 \,d\Pp
 \\ &\leq
 [\varphi]_{\alpha,\beta}
 \int
 \Delta_{\alpha,\beta}
 \bigl(f_\omega^n(x),f_\omega^n(y)\bigr)
 \,d\Pp  \leq
 C\rho^n[\varphi]_{\alpha,\beta}
 \Delta_{\alpha,\beta}(x,y),
\end{align*}
so $        [P^n\varphi]_{\alpha,\beta}
        \leq
        C\rho^n[\varphi]_{\alpha,\beta}$.
Applying the inequality with $x=x_0$ and integrating $y$ with respect to any stationary measure $\nu$ gives
\[ |P^n\varphi(x_0)-\int \varphi\,d\nu|\leq
 C\rho^n [\varphi]_{\alpha,\beta}
 \int \Delta_{\alpha,\beta}(x_0,y)\,d\nu
\]
with $\int \Delta_{\alpha,\beta}(x_0,y)\,d\nu\leq  V(x_0)^\beta \left(\int V\,d\nu\right)^\beta <\infty$ by Proposition \ref{prop:tightness}. 
It follows that
\[
 \bigl\|P^n\varphi-\int \varphi \,d\nu \cdot \1\bigr\|_{\alpha,\beta}\leq  C'\rho^n
 \|\varphi\|_{\alpha,\beta}.
\]
If $\nu'$ is another stationary probability, this estimate, applied
to bounded Lipschitz functions and integrated against $\nu'$, shows
that $\nu'=\nu$. This completes the proof.

\section{Compact repellers}\label{sec:repellers}
\label{subsec:compact-repellers}

Let $(M,d)$ be a compact metric space. Without loss of generality we may assume $d\leq1$.
Let $f$ be a random Lipschitz map of $M$. Let $E\subset M$ be a
nonempty closed set \emph{totally invariant} under $f$, that is,
$f_\omega^{-1}(E)=E$ for $\Pp$-almost every~$\omega$.
Put $X=M\setminus E$, and assume $X\neq\emptyset$.
For any Lipschitz map $g$ of $M$ leaving $E$ totally invariant, set
\[
l_E(g)=\Lip(g)+\sup_{x\notin E}\frac{d(x,E)}{d(g(x),E)}.
\]
In particular, if $g$ is bi-Lipschitz, then
$l_E(g)\leq l(g)\eqdef \Lip(g)+\Lip(g^{-1})$. 	In what follows, we assume $\int\log^+l_E(f_\omega)\,d\Pp<\infty$.

\begin{definition}\label{def:repelling-set}
	We say that $E$ is \emph{repelling} if, for every
	$x\in\partial E$ and all sufficiently large $n$,
	\[
	\limsup_{\substack{y\to x\\y\in X}}
	\int\log\frac{d(y,E)}{d(f_\omega^n(y),E)}\,d\Pp<0.
	\]
\end{definition}

\begin{rem} For $x\in\partial E$ we denote 
$$
\psi_{n}(x)=\limsup_{\substack{y\to x\\y\notin E}} \int\log\frac{d(y,E)}{d(f_\omega^n(y),E)}\,d\Pp.
$$
	\begin{enumerate}[label=\textup{(\roman*)}]
    	\item  $(\psi_{n})_{n\geq1}$ is a $P$-subadditive sequence of upper semicontinuous functions on the compact set $\partial E$. Indeed, writing $a_n(y)=\int\log\frac{d(y,E)}{d(f_\omega^n(y),E)}\,d\Pp$, we have $a_{n+m}=a_m+P^ma_n$ on~$X$. Since $a_n\leq n\int\log l_E(f_\omega)\,d\Pp$ and $f_\omega^m(\partial E)\subset\partial E$, reverse Fatou gives $\psi_{n+m}\leq\psi_{m}+P^m\psi_{n}$. Thus one might lighten the assumptions to have a repeller:
	 \begin{itemize}
	  	\item If, for every \(P\)-stationary measure \(\mu\) supported on
\(\partial E\), there is \(n\geq1\) such that
\(\int\psi_{n}\,d\mu<0\), then \(E\) is repelling by
Proposition~\ref{Kingman-noncompact}, applied on \(\partial E\).
	  	\item If, for every $x\in\partial E$, there is
	  	$n=n(x)$ such that $\psi_{n}(x)<0$, then $E$ is
	  	repelling by Lemma~\ref{partition} and compactness.
        \end{itemize}
	\item By upper semicontinuity, compactness of $\partial E$, and the subadditivity of $(\psi_{n})_{n\geq 1}$, it is equivalent to say that there is an integer $n\geq1$ and a neighborhood $U$ of $E$ such that 
	\[
	\sup_{y\in U\setminus E}
	\int\log\frac{d(y,E)}{d(f_\omega^n(y),E)}\,d\Pp<0.
	\]
	\end{enumerate}
\end{rem}

Under a stronger moment assumption, a repeller yields a Lyapunov
function for the dynamics on $X=M\setminus E$.

\begin{proposition}\label{prop:repeller-drift}
	Assume that $E$ is repelling and that
	$\int l_E(f_\omega)^\eta\,d\Pp<\infty$ for some $\eta>0$.
	For all sufficiently small $\gamma>0$, there is $N\geq1$ such that
	\begin{equation}\label{eq:repeller-V}
		V(x)\eqdef d(x,E)^{-\gamma}
	\end{equation}
	is a Lyapunov function for $P^N$. Moreover, there exists a Lyapunov function $V'$ for $P$ equivalent to $V$, in the sense that $V\leq V'\leq M V$ for some $M\geq1$.
\end{proposition}

\begin{proof}
By the repelling assumption, choose $N\geq1$, a neighborhood $U$ of $E$, and $a>0$ such that
\[
\sup_{y\in U\setminus E}\int Z_N(\omega,y)\,d\Pp\leq-a,
\qquad
Z_N(\omega,y)\eqdef \log\frac{d(y,E)}{d(f_\omega^N(y),E)}.
\]
The quantity $|Z_N(\omega,y)|$ is bounded, uniformly in $y\in X$, by
$\sum_{j=0}^{N-1}\log l_E(f_{\omega_j})$. This random bound has an exponential moment by the hypothesis and independence of the symbols. Applying
$e^t\leq1+t+\frac12t^2e^{|t|}$ to $t=\gamma Z_N(\omega,y)$ therefore gives, uniformly in $y$ as $\gamma\downarrow 0$,
\[
\int e^{\gamma Z_N(\omega,y)}\,d\Pp
\leq1+\gamma\int Z_N(\omega,y)\,d\Pp+O(\gamma^2).
\]
For sufficiently small $\gamma>0$, set $q\eqdef 1-a\gamma/2\in(0,1)$.
Then $P^NV(y)/V(y)\leq q$ for $y\in U\setminus E$, where
$V(y)=d(y,E)^{-\gamma}$. On $X\setminus U$, the function $V$ is bounded; also
$P^NV/V\leq\int l_E(f_\omega^N)^\gamma\,d\Pp<\infty$ on $X$.
Hence $P^NV\leq qV+C$ on $X$ for some $C<\infty$.

Set $r=q^{1/N}$ and $V'=\sum_{k=0}^{N-1}r^{-k}P^kV$. A direct calculation gives
\[
PV'=r\left(\frac{P^NV}{r^N}
+\sum_{k=1}^{N-1}\frac{P^kV}{r^k}\right)
\leq rV'+\frac{C}{r^{N-1}}.
\]
For each fixed $k$, the bound
$V(f_\omega^k(x))\leq V(x)l_E(f_\omega^k)^\gamma$
shows that $P^kV/V$ is bounded. It also gives an integrable domination
locally uniform in $x\in X$; hence $P^kV$ is continuous by dominated
convergence. Consequently $V'$ is continuous, $V\leq V'\leq MV$ for some
$M<\infty$, and $V'$ is proper since $V$ is proper.
\end{proof}

If the restricted map on $X=M\setminus E$ is mostly contracting,
we can apply Theorems~\ref{thm:quasicompact}
and~\ref{thm:product-weighted-global-contraction} with $V'$.
Since $V\leq V'\leq CV$, the corresponding weighted functions are
uniformly comparable, and the conclusions also hold with $V$ in
place of $V'$. In this case,
$\Delta_{\alpha,\beta}(x,y)
=d(x,y)^\alpha/[d(x,E)^{\gamma\beta}d(y,E)^{\gamma\beta}]$.
Renaming $\gamma\beta$ as $\beta$, the corresponding weighted
H\"older space $\Hol_{\alpha,\beta}(X)$ has norm equivalent to
\[
\|\varphi\|_{\alpha,\beta}
\asymp |\varphi(x_0)|+
\sup_{x\neq y}\frac{|\varphi(x)-\varphi(y)|d(x,E)^{\beta}d(y,E)^{\beta}}{d(x,y)^\alpha}.
\]

This gives the following result.

\begin{corollary}\label{cor:repeller-quasicompact}
	Let $f$ be a random Lipschitz map of a compact metric space
	$(M,d)$, and let $E\subset M$ be a totally invariant closed set.
	Assume $\int l_E(f_\omega)^\eta\,d\Pp<\infty$ for some $\eta>0$,
	that $E$ is repelling, and that $\lambda(\mu)<0$ for every
	stationary probability $\mu$ with $\mu(E)=0$.
	Then, for all sufficiently small $\alpha,\beta>0$ with
	$\alpha/\beta$ sufficiently small, $P$ is quasi-compact on
	$\Hol_{\alpha,\beta}(X)$, where $X=M\setminus E$.
	If $f|_X$ is also stationarily proximal, then, for some $C>0$
	and $\rho\in(0,1)$,
\[
\int\frac{d(f_\omega^n(x),f_\omega^n(y))^\alpha}
 {d(f_\omega^n(x),E)^\beta d(f_\omega^n(y),E)^\beta}\,d\Pp
 \leq C\rho^n\frac{d(x,y)^\alpha}{d(x,E)^\beta d(y,E)^\beta}, \quad \text{for every $x,y\in X$ and $n\geq0$,}
\]
and $P$ has a spectral gap on $\Hol_{\alpha,\beta}(X)$.
\end{corollary}

\section{Limit theorems}\label{sec:limit-theorems}
Throughout this section, $f$ satisfies the assumptions of
Theorem~\ref{thm:product-weighted-global-contraction}; in particular,
it has a unique stationary probability $\nu$. A \emph{one-step random
observable} is a real-valued function $u:\Omega\times X\to\R$ of
the form $u_\omega=u_{\omega_0}$. For a fixed $x\in X$, we study
the Birkhoff sums
\[
S_n=\sum_{k=0}^{n-1}u_{\sigma^k\omega}(f_\omega^k(x)).
\]
When the initial point varies, we write $S_n(\omega,x)$.

\begin{definition}\label{def:observable-nonarithmetic} Let $u$ be a one-step random observable. 
	\begin{enumerate}
		\item  \(u\) is \emph{degenerate} if, there are  $\alpha,\beta>0$,  
		\(g\in\Hol_{\alpha,\beta}(X)\) real-valued and $m\in \R$ such that
\begin{equation}\label{eq:observable-additive-coboundary}
			u_\omega=m+(g-g\circ f_\omega)
			\quad \Pp\times\nu\text{-a.e.}.
		\end{equation}
					It is \emph{non-degenerate} otherwise.
		\item  \(u\) is \emph{arithmetic} if,  there are 
		\(t\in\R\setminus\{0\}\), \(\lambda\in\C\) with \(|\lambda|=1\) and  $\alpha,\beta>0$,
		\(h\in\Hol_{\alpha,\beta}(X)\) such that \(|h|=1\) \(\nu\)-a.e. and
		\begin{equation}\label{eq:observable-circle-coboundary}
			e^{it u_\omega}=\lambda \frac{h}{h\circ f_\omega}
			\quad \Pp\times\nu\text{-a.e.}.
		\end{equation}
		It is \emph{non-arithmetic} otherwise.
	\end{enumerate}
\end{definition}	
If $u$ is degenerate, then it is arithmetic: take $h=e^{itg}$ for
any $t\neq0$ and $\lambda=e^{itm}$. Conversely, taking a measurable
argument in the arithmetic case gives a measurable function $g$,
defined $\nu$-almost everywhere, such that
$u_\omega+g\circ f_\omega-g\in a+b\mathbb Z$ for
$\Pp\times\nu$-almost every~$(\omega,x)$, for suitable $a\in\R$
and $b>0$.

For $\alpha>0$, write
\[
 [h]_\alpha\eqdef\sup_{x\neq y}\frac{|h(x)-h(y)|}{d(x,y)^\alpha},
 \qquad
 \|h\|_\alpha\eqdef\|h\|_\infty+[h]_\alpha.
\]
We denote by $C^\alpha(X;\C)$ the bounded $\alpha$-H\"older
functions with this norm. The goal of this section is to prove the following theorem:

\begin{theorem}\label{thm:abstract-limit-theorems}	
Under the standing assumptions on $f$, let $u:\Omega\times X\to\R$
be a random observable such that $u_\omega$ is bounded and H\"older
and, for some $\eta,\alpha_0>0$, satisfies
\begin{equation}\label{eq:observable-mixed-moment}
	\int (1+[u_\omega]_{\alpha_0})^2
	e^{\eta\|u_\omega\|_\infty}\,d\Pp<\infty.
\end{equation}
Fix $x_0\in X$ and set
$S_n(\omega)=\sum_{k=0}^{n-1}
u_{\sigma^k\omega}(f_\omega^k(x_0))$ and
$m=\int u(\omega,x)\,d\Pp\,d\nu$. Then
\begin{enumerate}[label=\textup{(\roman*)}]
\item \emph{Central limit theorem.}
\((S_n-nm)/\sqrt n\) converges in law to
\(\cN(0,\sigma^2)\) for some $\sigma\geq0$.
If $u$ is non-degenerate, then $\sigma>0$.
\item \emph{Berry--Esseen bound.} If \(u\) is non-degenerate, then 
\(d_{\mathrm K}\bigl((S_n-nm)/(\sigma\sqrt n),
\cN(0,1)\bigr)=O(n^{-1/2})\) as $n\to \infty$,
where $d_{\mathrm K}$ is the Kolmogorov distance.
\item \emph{Local central limit theorem.} If \(u\) is non-arithmetic,
then, for every
\(g\in C_c(\R)\),
\[
\sup_{a\in\R}\left|
\sigma\sqrt{2\pi n}\int g(S_n-nm-a)\,d\Pp
-e^{-a^2/(2n\sigma^2)}\int_\R g(s)\,ds\right|\longrightarrow0.
\]
\item \emph{Renewal theorem.} If \(m>0\) and
\(u\) is non-arithmetic, then, for
\(g\in C_c(\R)\),
\[
\sum_{n\geq1}\int g(S_n-a)\,d\Pp
\longrightarrow\frac1m\int_\R g(s)\,ds
\quad\text{as }a\to +\infty,
\]
and this sum tends to zero as \(a\to -\infty\).
\item \emph{Large deviations.} 
If \(u\) is non-degenerate, there are
\(\varepsilon_0>0\) and a strictly convex nonnegative function
\(c\) near zero, independent of $x_0$ and with \(c(0)=c'(0)=0\) and $c''(0)=1/\sigma^2$, such that for
\(0<\varepsilon<\varepsilon_0\),
\[
\qquad \quad \lim_{n\to \infty}\frac1n\log\Pp\bigl(
S_n-nm>n\varepsilon\bigr)=-c(\varepsilon),
\quad
\lim_{n\to \infty}\frac1n\log\Pp\bigl(
S_n-nm<-n\varepsilon\bigr)=-c(-\varepsilon).
\]
For every $u$, including the degenerate case, there are
$\varepsilon_1,h>0$ and $C_{x_0}>0$ such that, for all $n\geq1$ and
$0<\varepsilon<\varepsilon_1$,
\(\Pp(|S_n-nm|>n\varepsilon)
\leq C_{x_0}e^{-hn\varepsilon^2}\).
\end{enumerate}
\end{theorem}

\subsection{Analyticity of the Fourier--Laplace operators} Define the Fourier--Laplace operator
\[
P_z\varphi(x)=\int e^{zu_\omega(x)}
\varphi(f_\omega(x))\,d\Pp.
\]

\begin{proposition}\label{prop:analyticity}
If $\alpha$ and $\beta$ are sufficiently small, the map \(z\mapsto P_z\)
is analytic on the strip
\(\{z\in\C:|\Ree z|<\eta/2\}\) into
\(\mathcal L(\Hol_{\alpha,\beta}(X))\).
\end{proposition}

In order to prove this proposition we begin by the following lemma.
\begin{lemma}\label{lem:weighted-observable-operator}
Let \(h_\omega\in C^\alpha(X;\C)\) be a measurable one-step family
such that \(\int\|h_\omega\|_\alpha^2\,d\Pp<\infty\), and put
\(Q_h\varphi=\int h_\omega(\varphi\circ f_\omega)\,d\Pp\).
Then \(Q_h:\Hol_{\alpha,\beta}(X)\to\Hol_{\alpha,\beta}(X)\) is bounded, with
\[\|Q_h\|\leq C(\int\|h_\omega\|_\alpha^2\,d\Pp)^{1/2}.\]
\end{lemma}
\begin{proof}
For \(x\neq y\), write
\[
\begin{aligned}
&|h_\omega(x)\varphi(f_\omega(x))
   -h_\omega(y)\varphi(f_\omega(y))| \leq \\
&\quad\leq |h_\omega(x)|
 |\varphi(f_\omega(x))-\varphi(f_\omega(y))|
 +|\varphi(f_\omega(y))||h_\omega(x)-h_\omega(y)|.
\end{aligned}
\]
Put \(R_\omega(x)=V(f_\omega(x))/V(x)\). Dividing the first term
by \(\Delta_{\alpha,\beta}(x,y)\) bounds it by
\[
\|\varphi\|_{\alpha,\beta}\|h_\omega\|_\alpha
\Lip(f_\omega)^\alpha R_\omega(x)^\beta R_\omega(y)^\beta.
\]
For the second term, use
\(|h_\omega(x)-h_\omega(y)|
 \leq\|h_\omega\|_\alpha d(x,y)^\alpha\) and
\(|\varphi(f_\omega(y))|
 \leq C\|\varphi\|_{\alpha,\beta}V(f_\omega(y))^\beta\).
Since \(V(x)\geq1\), its quotient is at most
\(C\|\varphi\|_{\alpha,\beta}\|h_\omega\|_\alpha R_\omega(y)^\beta\).
Integrating in \(\omega\), Cauchy--Schwarz and the fixed-time moment
bounds from the proof of Proposition~\ref{prop:contractions}\textup{(i)}
give
\[
[Q_h\varphi]_{\alpha,\beta}
\leq C'\|\varphi\|_{\alpha,\beta}
\left(\int\|h_\omega\|_\alpha^2\,d\Pp\right)^{1/2}.
\]
Indeed, the required uniform second moments of
\(\Lip(f_\omega)^\alpha R_\omega(x)^\beta R_\omega(y)^\beta\)
and \(R_\omega(y)^\beta\) follow from
\(6\alpha\leq\eta_0\), \(6\beta<1\), and
\(\sup_{x\in X} PV/V<\infty\).
The estimate at \(x_0\) follows in the same way from
\(|\varphi(f_\omega(x_0))|
 \leq C\|\varphi\|_{\alpha,\beta}V(f_\omega(x_0))^\beta\)
and \(2\beta\leq1\).
\end{proof}

\begin{proof}[Proof of Proposition~\ref{prop:analyticity}]
Fix an imaginary number \(z_0=it_0\). We prove that near \(z_0\),
\(P_z\) has the analytic expansion
\(\sum_{k\geq0}(z-z_0)^kP_k/k!\), where
\[
P_k\varphi=\int u_\omega^k e^{it_0u_\omega}
(\varphi\circ f_\omega)\,d\Pp.
\]
Put \(U_\omega=\|u_\omega\|_\infty\) and
\(H_\omega=[u_\omega]_\alpha\). Choose \(\alpha\leq\alpha_0\); since \(d\leq1\), one has \(H_\omega\leq[u_\omega]_{\alpha_0}\). The lemma gives
\(\|P_k\|\leq C(\int
\|u_\omega^k e^{it_0u_\omega}\|_\alpha^2\,d\Pp)^{1/2}\).
The product rule and \([e^{it_0u_\omega}]_\alpha
\leq |t_0|H_\omega\) give
\[
\|u_\omega^k e^{it_0u_\omega}\|_\alpha
\leq U_\omega^k+
\bigl(kU_\omega^{k-1}+|t_0|U_\omega^k\bigr)H_\omega,
\]
where the term with \(kU_\omega^{k-1}\) is omitted for \(k=0\).
Choose \(0<r<s<\eta/2\). Applying \(t^j\leq j!e^t\)
with \(t=sU_\omega\) and \(j=k,k-1\) gives
\[
\|u_\omega^k e^{it_0u_\omega}\|_\alpha
\leq \frac{k!}{s^k}e^{sU_\omega}
\bigl(1+(s+|t_0|)H_\omega\bigr).
\]
Consequently, by~\eqref{eq:observable-mixed-moment},
\[
\frac{\|P_k\|}{k!}\leq C_s s^{-k},\qquad
C_s=C\left(\int e^{2sU_\omega}
(1+(s+|t_0|)H_\omega)^2\,d\Pp\right)^{1/2}<\infty.
\]
Thus \(\sum_{k\geq0}r^k\|P_k\|/k!<\infty\), and the series
converges in operator norm for \(|z-it_0|<r\).
For every \(x\) and \(\varphi\), the same bounds, together with
\(|\varphi(f_\omega(x))|
 \leq C\|\varphi\|_{\alpha,\beta}V(f_\omega(x))^\beta\),
allow us to interchange sum and integral. Its sum is \(P_z\).
As \(t_0\) is arbitrary and \(r\) can be taken as close to
\(\eta/2\) as desired, these disks cover the stated strip.
\end{proof}

\subsection{Essential spectral radius and arithmeticity}
\begin{proposition}\label{prop:twisted-essential-radius}
If $\alpha/\beta$ and $\beta$ are sufficiently small, then
$r_{\mathrm{ess}}(P_{it})<1$ and $r(P_{it})\leq1$ for every $t\in\R$. Moreover, if $u$ is
non-arithmetic, then $r(P_{it})<1$ for every $t\neq0$.
\end{proposition}
\begin{proof}

Fix $\alpha,\beta>0$ and $x,y\in X$. For $n\geq0$,
\[
P_{it}^n\varphi(x)-P_{it}^n\varphi(y)
=\int\bigl(e^{itS_n(\omega,x)}\varphi(f_\omega^n(x))
-e^{itS_n(\omega,y)}\varphi(f_\omega^n(y))\bigr)\,d\Pp.
\]

The elementary inequality
$|e^{i\theta_1}z_1-e^{i\theta_2}z_2|
\leq|z_1-z_2|+|\theta_1-\theta_2||z_1z_2|^{1/2}$
gives $|P_{it}^n\varphi(x)-P_{it}^n\varphi(y)|\leq A_n+B_n$,
where
\[
\begin{aligned}
A_n&=\int|\varphi(f_\omega^n(x))-\varphi(f_\omega^n(y))|\,d\Pp,\\
B_n&=|t|\int|S_n(\omega,x)-S_n(\omega,y)|
|\varphi(f_\omega^n(x))\varphi(f_\omega^n(y))|^{1/2}\,d\Pp.
\end{aligned}
\]

For sufficiently small $\beta$ and $\alpha/\beta$,
Theorem~\ref{thm:product-weighted-global-contraction} gives
$A_n\leq C\rho^n[\varphi]_{\alpha,\beta}
\Delta_{\alpha,\beta}(x,y)$. Set
$R_{n,\omega}(x)=V(f_\omega^n(x))/V(x)$. If $4\beta\leq1$,
 \begin{align*}
	B_n&\leq |t|\,|\varphi|_{2\beta}\Delta_{\alpha,\beta}(x,y)
	\int [S_n(\omega,\cdot)]_\alpha
	R_{n,\omega}(x)^\beta R_{n,\omega}(y)^\beta\,d\Pp \leq C_{t,n}|\varphi|_{2\beta}\Delta_{\alpha,\beta}(x,y),
	\end{align*}
where $C_{t,n}=|t|C_V^2
\bigl(\int [S_n(\omega,\cdot)]_\alpha^2\,d\Pp\bigr)^{1/2}$
is finite. Indeed,
\[
[S_n(\omega,\cdot)]_\alpha
\leq\sum_{k=0}^{n-1}[u_{\sigma^k\omega}]_\alpha
\Lip(f_\omega^k)^\alpha.
\]
For each $k$, the two factors are independent;
$\int [u_{\sigma^k\omega}]_\alpha^2\,d\Pp<\infty$
by~\eqref{eq:observable-mixed-moment}, and
$\int\Lip(f_\omega^k)^{2\alpha}\,d\Pp<\infty$
when $2\alpha<\eta_0$. Consequently,
\begin{equation}\label{eq:twisted-LY}
[P_{it}^n\varphi]_{\alpha,\beta}
\leq C\rho^n[\varphi]_{\alpha,\beta}
+C_{t,n}|\varphi|_{2\beta}.
\end{equation}
Fixing $N$ with $C\rho^N\leq1/2$, we have a Lasota--Yorke inequality $[P_{it}^{N}\varphi]_{\alpha,\beta}\leq 
 \frac{1}{2}[\varphi]_{\alpha,\beta}
+C'|\varphi|_{2\beta}$. By the drift estimate and $2\beta\leq1$, the weaker norms are uniformly bounded: $|P_{it}^n\varphi|_{2\beta}\leq |P^n(|\varphi|)|_{2\beta}\leq C_V^{2\beta}|\varphi|_{2\beta}$. The embedding $\Hol_{\alpha,\beta}(X)\hookrightarrow C_{2\beta}(X)$ is compact. Then the Ionescu--Tulcea--Marinescu theorem gives $r_{\mathrm{ess}}(P_{it})<1$ and $r(P_{it})\leq1$; in particular, if $r(P_{it})=1$, it has an eigenvalue $\lambda$ with modulus $1$.

In this case, there is $h\in\Hol_{\alpha,\beta}(X)$ with
$P_{it}h=\lambda h$. Thus $P|h|\geq|h|$. These functions have the
same integral against $\nu$, so they agree $\nu$-almost everywhere.
Ergodicity of $\nu$ then makes $|h|$ constant $\nu$-almost everywhere.
This constant cannot be zero: otherwise,
$|h|\leq P^n|h|\to\int|h|\,d\nu=0$ in the strong norm, contradicting
$h\neq0$. Rescale $h$ so that $|h|=1$ $\nu$-almost everywhere.
Equality in the triangle inequality for $P_{it}h=\lambda h$ now gives
$e^{itu_\omega}h\circ f_\omega=\lambda h$ for
$\Pp\times\nu$-almost every~$(\omega,x)$. Hence $u$ is arithmetic.
\end{proof}

The following proposition shows that the measurable non-arithmeticity
condition used in the introduction for matrix products agrees with
Definition~\ref{def:observable-nonarithmetic}.
\begin{proposition}\label{prop:measurable-arithmeticity}
Under the assumptions of Theorem~\ref{thm:abstract-limit-theorems},
$u$ is arithmetic if and only if there exist $a\in\R$, $b>0$, and a
measurable real-valued function $g$, defined $\nu$-almost everywhere,
such that
\[
u_\omega(x)+g(f_\omega(x))-g(x)\in a+b\mathbb Z
\quad\text{for $\Pp\times\nu$-almost every~$(\omega,x)$.}
\]
\end{proposition}
\begin{proof}
If $u$ is arithmetic, take a measurable argument of the function $h$
in~\eqref{eq:observable-circle-coboundary}. Dividing it by $t$ gives
$g$ with the stated property, with $b=2\pi/|t|$ and an appropriate $a$.
Conversely, given $a,b,g$ as above, set $t=2\pi/b$,
$h=e^{it g}$, and $\zeta=e^{it a}$. Then
$e^{it u_\omega}h\circ f_\omega=\zeta h$ almost everywhere, so
$P_{it}h=\zeta h$ in $L^2(\nu)$.

Suppose that $u$ were non-arithmetic. For admissible $\alpha,\beta$
with $2\beta\leq1$, Proposition~\ref{prop:twisted-essential-radius}
gives $r(P_{it}|\Hol_{\alpha,\beta}(X))<1$.
The space $\Hol_{\alpha,\beta}(X)$ embeds continuously and densely
into $L^2(\nu)$: the first assertion follows from
$\int V\,d\nu<\infty$, and the second from the density of bounded
Lipschitz functions. By stationarity, $P_{it}$ is a contraction on
$L^2(\nu)$. Its powers therefore converge to zero on that space,
first on the dense H\"older subspace and then everywhere. This
contradicts $P_{it}^nh=\zeta^nh$. Thus $u$ is arithmetic.
\end{proof}

\subsection{Proof of Theorem~\ref{thm:abstract-limit-theorems}}
We apply Hennion--Herv\'e~\cite[Thms.~A--E]{HH01}. Since $u$ may be random, we use
a Markov chain whose state includes
the current random symbol. Add a symbol \(e\) of probability zero to
\(T\), with \(f_e=\mathrm{id}_X\) and \(u_e=0\), and put
\(\mathcal E=T\times X\),
\(\bar f_s(t,x)=(s,f_t(x))\), and \(v(t,x)=u_t(x)\).
For \(\varphi\in\Hol_{\alpha,\beta}(X)\), set
\(J\varphi(t,x)=\varphi(f_t(x))\); the value at \((e,x)\) shows that
\(J\) is an isometry when its image is equipped with the transported
norm. Define the Fourier--Laplace kernel of the extended chain using the
observable at the arrival state:
\[
 Q_z\psi(t,x)=\int e^{zu_s(f_t(x))}\psi(s,f_t(x))\,dp(s).
\]
Then
\[
Q_zJ\varphi=J(P_z\varphi),\qquad
Q_z^n\1(e,x_0)=P_z^n\1(x_0)
=\int e^{zS_n(\omega)}\,d\Pp.
\]
The extended chain has stationary probability \(p\times\nu\), whose
functional on \(\mathcal{B}=J(\Hol_{\alpha,\beta}(X))\) is continuous because
\(\int V^\beta\,d\nu<\infty\). This Banach space contains the
constants, is closed under complex conjugation and modulus, and the
product of two bounded elements belongs to it. Point evaluations are
continuous. Theorem~\ref{thm:product-weighted-global-contraction}
gives the spectral hypotheses of~\cite{HH01}, while
Proposition~\ref{prop:analyticity} gives its Fourier differentiability
and the required Fourier and Laplace kernels on the whole real axis
and a complex strip. Thus Theorems A and B give (i)--(ii).

By~\cite[Thm.~A\textup{(iii)}]{HH01}, the equality
\(\sigma^2=0\) gives an additive coboundary in the Banach space of the
extended chain. Transferring it through the isometry \(J\) yields
\eqref{eq:observable-additive-coboundary}; hence non-degeneracy implies
\(\sigma>0\). In views of Proposition~\ref{prop:twisted-essential-radius}, under our assumption of  non-arithmeticity,~\cite[Thm.~C]{HH01}  gives (iii),
and \cite[Thm.~D]{HH01} concludes (iv). Finally, \cite[Thm.~E]{HH01}, applied to $u-m$ and
$m-u$, gives the limits in (v). Since \cite[Thm.~E]{HH01} uses
the same rate function for every admissible initial distribution,
$c$ does not depend on $x_0$. Moreover,  the conditions $c(0)=c'(0)=0$, and $c''(0)=1/{\sigma^2}$ follows~\cite[Prop.~VIII.3]{HH01}.

The upper bound holds also when $\sigma=0$. Indeed, for real $t$ near
zero, the simple leading eigenvalue $\lambda(t)$ of $P_t$ satisfies
$\lambda(0)=1$ and $(\log\lambda)'(0)=m$. Analyticity and the uniform
spectral decomposition therefore give $K,t_0>0$ and $C_{x_0}<\infty$
such that, for $n\geq1$ and $0\leq t\leq t_0$,
\[
\int e^{\pm t(S_n-nm)}\,d\Pp
=e^{\mp nmt}P_{\pm t}^n\1(x_0)
\leq C_{x_0}e^{Knt^2}.
\]
Thus, choosing $t=\varepsilon/(2K)$ and decreasing $\varepsilon_1$
if necessary, Chernoff's bound gives
\[
\Pp(|S_n-nm|>n\varepsilon)
\leq 2C_{x_0}e^{-n(t\varepsilon-Kt^2)}
=2C_{x_0}e^{-n\varepsilon^2/(4K)}.
\]

\subsection{Almost sure invariance principle}\label{ss:asip} 
A real-valued process $(S_n)_{n\geq1}$ satisfies an \emph{almost sure invariance principle} (ASIP) with mean $m$, variance $\sigma^2$, and error $O(a_n)$ if, without changing its joint law, it can be realized together with independent $G_k\sim\cN(0,\sigma^2)$ so that
\[
        S_n-nm-\sum_{k=1}^nG_k=O(a_n)\qquad\text{a.s.}
\]
The meaning of an $o(a_n)$ error is analogous.

\begin{theorem}
\label{cor:ASIP-global-contraction}
Under the assumptions of
Theorem~\ref{thm:product-weighted-global-contraction}, fix admissible
$\alpha,\beta>0$, and let $\nu$ be the unique stationary probability.
Let $u$ be a real-valued one-step observable such that, 
\begin{equation}\label{eq:asip-final-observable-regularity}
 \int |u_\omega(x)-u_\omega(y)|\,d\Pp
 \leq L\Delta_{\alpha,\beta}(x,y)
 \qquad \text{for $x,y\in X$ and some $L<\infty$}.
\end{equation}
Suppose that, for some $\gamma\in(0,\infty)$ and $a>0$, $\int e^{a|u_\omega|^\gamma}\,d\Pp\,d\nu<\infty$, or say that $\gamma=+\infty$ if $u$ is essentially bounded under $\Pp\times\nu$. Then, for every fixed $x\in X$, the process $(S_n(x))_{n\geq1}$ satisfies an ASIP with mean $m=\int u\,d\Pp\,d\nu$, variance $\sigma^2\geq0$, and error $O((\log n)^{2+1/\gamma})$, where
$1/\infty=0$. Moreover,
\[
 \sigma^2=\lim_{n\to \infty}\frac1n
 \int (S_n(\omega,y)-nm)^2\,d\Pp\,d\nu.
\]
\end{theorem}

\begin{remark} \label{rem:variance_ASIP}
The mixed moment assumption~\eqref{eq:observable-mixed-moment} of
Theorem~\ref{thm:abstract-limit-theorems} implies the hypotheses above
with $\gamma=1$, giving an ASIP with error $O((\log n)^3)$. The mean
$m$ and the asymptotic variance $\sigma^2$ agree in both theorems.
\end{remark}

\begin{lemma}\label{lem:asip-bounded-comparison}
Let $U=(U_n)_{n\geq1}$ and $V=(V_n)_{n\geq1}$ be real-valued
processes. Assume that $U$ satisfies an ASIP with mean $m$ and error $O(a_n)$,
where $\inf_{n\geq1} a_n>0$ (respectively, with error $o(a_n)$, where
$a_n\to \infty$). Suppose that processes $U'$ and $V'$ can be
realized on a common probability space with
$U'\stackrel{\mathrm{law}}=U$, $V'\stackrel{\mathrm{law}}=V$, and
$\sup_{n\geq1}|V_n'-U_n'|<\infty$ almost surely.
Then $V$ satisfies an ASIP with the same mean, variance, and error.
\end{lemma}

\begin{proof}
Let $\pi$ be the joint law of $(U',V')$ on $\mathcal E^2$, where
$\mathcal E=\R^\N$, and let $\widehat\pi$ be the joint law of a copy
of $U$ and the Gaussian walk supplied by its ASIP. Both laws have the
same first marginal $\mu_U$. Disintegrate them as
$\pi(du,dv)=\mu_U(du)K(u,dv)$ and
$\widehat\pi(du,db)=\mu_U(du)L(u,db)$, and glue them by
$\mu_U(du)K(u,dv)L(u,db)$. On this space,
\[
v_n-nm-b_n=(v_n-u_n)+(u_n-nm-b_n).
\]
The first term is almost surely bounded, which preserves the stated
$O(a_n)$ or $o(a_n)$ error.
\end{proof}

\begin{proof}[Proof of Theorem~\ref{cor:ASIP-global-contraction}]
Work on $(\Omega\times X^2,\Pp\times\nu\times\nu)$, with
coordinates $(\omega,y,y^*)$. Write $Y=y$ and $Y^*=y^*$ for the
initial conditions, which are independent with law $\nu$ and
independent of $\omega$.
We may take $\varepsilon_0\leq1$. Since $d\leq1$, $\beta\leq1$ and
$\int V\,d\nu<\infty$, both
$D_\nu=\int\Delta_{\alpha,\beta}(y,y^*)\,d\nu(y)\,d\nu(y^*)$
and $D_x=\int\Delta_{\alpha,\beta}(x,y)\,d\nu$ are finite.
Set $Y_k=u_{\omega_{k-1}}(f_\omega^{k-1}(Y))-m$, and define
$Y_k^*$ similarly with $Y^*$ and the same symbols. Then $(Y_k)_{k\geq1}$ is
stationary and centered. The independence of $\omega_{k-1}$ from the
trajectories up to time $k-1$, the averaged regularity
assumption~\eqref{eq:asip-final-observable-regularity}, and
\eqref{eq:product-global-contraction} give, for $k\geq1$,
\[
 \delta(k)\eqdef \int|Y_k-Y_k^*|\,d\Pp\,d\nu(y)\,d\nu(y^*)
 \leq LC\rho^{k-1}D_\nu.
\]
Moreover, for finite $\gamma$, the assumed moment implies
$\int e^{a'|Y_1|^\gamma}\,d\Pp\,d\nu<\infty$ for some $a'>0$;
for $\gamma=\infty$, $Y_1$ is essentially bounded. Thus the tail
condition of~\cite[Thm.~1]{CDM24} has exponent $\gamma$, while its
coupling condition has exponent $1$ (rescale $Y_k$ by a constant to
absorb the prefactor in the bound for $\delta(k)$). The theorem gives
the stationary ASIP for $S_n(\omega,Y)-nm$ with error
$O((\log n)^{2+1/\gamma})$, and the formula for $\sigma^2$.

For a fixed $x\in X$, use the same symbols to couple the stationary
process $S_n(\omega,Y)$ with $S_n(\omega,x)$. By Tonelli and the same
estimate,
\[
 \int\sum_{k\geq0}
 \bigl|u_{\omega_k}(f_\omega^k(x))
       -u_{\omega_k}(f_\omega^k(y))\bigr|\,d\Pp\,d\nu
 \leq \frac{LC}{1-\rho}D_x<\infty.
\]
Consequently $\sup_{n\geq1}|S_n(\omega,x)-S_n(\omega,Y)|<\infty$ almost
surely. 
Lemma~\ref{lem:asip-bounded-comparison} transfers the stationary
ASIP to $S_n(x)=S_n(\cdot,x)$ with the same variance and polylogarithmic error.
\end{proof}

\begin{remark}
The semi-exponential tail is needed only for the polylogarithmic rate. Keeping the averaged regularity assumption, but requiring only
$u\in L^r(\Pp\times\nu)$ for some
$r>2$, \cite[Thm.~2 and Rem.~4]{CDM18} gives for every $2<q<r$ an ASIP
with error $o(n^{1/q})$.
\end{remark}

\section{Linear cocycles} \label{s:linear}
In this section, we study random matrix products and retain the
notation of the introduction. Thus $E$ is the equator of the cocycle
$A:\Omega\to\operatorname{GL}(m,\R)$, and
$f:\Omega\times P(\R^m)\to P(\R^m)$ is its projective random map.
We also put
\[
 \ell(\omega)\eqdef\max\{1,\|A(\omega)\|,\|A(\omega)^{-1}\|\}.
\]

\subsection{Global contraction and spectral gap}

\begin{proposition}
\label{prop:equator-repeller}
Assume~\eqref{eq:exp-moment} and
\(\lambda_1(A)>\lambda_2(A)\). If \(E\neq\{0\}\), then \(P(E)\) is a compact set completely invariant and repelling for \(f\) in the sense of
Definition~\ref{def:repelling-set}. 
\end{proposition}

\begin{proof}
Assume that \(E\neq\{0\}\). Since \(E\) is \(A\)-invariant and
\(A(\omega)\) is invertible, \(A(\omega)E=E\) for \(\Pp\)-a.e.~\(\omega\). Thus \(P(E)\) is compact and completely invariant.

By~\cite[Lem.~5.6.1]{BMbook},
$\Lip(f_\omega),\Lip(f_\omega^{-1})\leq\ell(\omega)^4$.
Since $P(E)$ is totally invariant, the definition of $l_{P(E)}$ gives
\[
l_{P(E)}(f_\omega)
\leq\Lip(f_\omega)+\Lip(f_\omega^{-1})
\leq 2\ell(\omega)^4.
\]
Thus~\eqref{eq:exp-moment} implies
$\int l_{P(E)}(f_\omega)^\eta\,d\Pp<\infty$
for $0<\eta\leq\eta_0/4$, as required by
Proposition~\ref{prop:repeller-drift}.

Put $\lambda_E\eqdef\lambda_1(A|_E)$. 
By the definition of the equator, $\lambda_E<\lambda_1(A)$. Let \(Q:\R^m\to\tilde{V}\) be the canonical quotient map, equip
\(\tilde{V}=\R^m/E\) with the quotient Euclidean norm, and let \(\tilde{A}\) denote
the cocycle induced by \(A\) on \(\tilde{V}\). According to~\cite[Claim 5.5.2]{BMbook}, 
$\lambda_1(\tilde{A})=\lambda_1(A)$,
and \(\tilde{A}\) is quasi-irreducible.
For $\hat x\in X$, we also write $Q\hat x\eqdef\widehat{Qx}$.

For \(\hat x\in P(\R^m)\) and
\(\hat v\in P(\tilde{V})\), define
\[
 a_n(\hat x)
 \eqdef
 \int\log\frac{\|A^n(\omega)x\|}{\|x\|}\,d\Pp,
 \qquad
 b_n(\hat v)
 \eqdef
 \int\log\frac{\|\tilde{A}^n(\omega)v\|}{\|v\|}\,d\Pp.
\]
The exponential moment implies that these functions are well defined and
continuous. 
Moreover, the sequences \((a_n|_{P(E)})_{n\geq1}\) and \((b_n)_{n\geq1}\) are additive for the
projective Markov operators induced by \(A|_E\) and \(\tilde{A}\),
respectively. 
Set $\theta=(\lambda_E+\lambda_1(A))/2$.
By the variational formula~\cite[Thm.~5.2.1]{BMbook}, every
stationary average of the projective potential on $P(E)$ is at most
$\lambda_E<\theta$. 
For every stationary probability $\tilde\mu$ on
$P(\tilde V)$, quasi-irreducibility gives~\cite[Claim~5.5.2]{BMbook}
\[
\int b_1(\hat v)\,d\tilde\mu
=\int
\log\frac{\|\tilde A(\omega)v\|}{\|v\|}
\,d\Pp\,d\tilde\mu
=\lambda_1(\tilde A)=\lambda_1(A)>\theta.
\]
Applying the compact uniform Kingman
theorem~\cite[Thm.~C]{BMbook} to the additive sequences
$(a_n|_{P(E)}-n\theta)_{n\geq1}$ and $(n\theta-b_n)_{n\geq1}$, respectively,
we obtain, for every sufficiently large $n$,
\[
 \max_{\hat x\in P(E)}a_n(\hat x)<n\theta
 <\min_{\hat v\in P(\tilde V)}b_n(\hat v).
\]

We now identify the main term from
Definition~\ref{def:repelling-set}. To do this, for every \(x\in\R^m\setminus E\), the definition of the angular distance
and the orthogonal decomposition \(\R^m=E\oplus E^\perp\) give
\[
 d(\hat x,P(E))
 =
 \frac{\|Qx\|}{\|x\|},
 \qquad
 \frac{d(\hat x,P(E))}
 {d(f_\omega^n(\hat x),P(E))}
 =
 \frac{\|A^n(\omega)x\|/\|x\|}
 {\|\tilde{A}^n(\omega)Qx\|/\|Qx\|}.
\]
Thus,  for
\(\hat y\in X\),
\begin{equation}
\label{eq:linear-boundary-Hn}
 \int
 \log\left(
 \frac{d(\hat y,P(E))}
 {d(f_\omega^n(\hat y),P(E))}
 \right)\,d\Pp
 =
 a_n(\hat y)-b_n(Q\hat y).
\end{equation}

For every sufficiently large $n$ and every
$\hat x\in P(E)$, continuity of $a_n$ and
\eqref{eq:linear-boundary-Hn} give
\[
 \limsup_{\substack{\hat y\to\hat x\\\hat y\in X}}
 \int\log\frac{d(\hat y,P(E))}{d(f_\omega^n(\hat y),P(E))}
 \,d\Pp
 \leq a_n(\hat x)-\min_{\hat v\in P(\tilde V)}b_n(\hat v)<0.
\]
This is precisely the repelling condition in
Definition~\ref{def:repelling-set}.
\end{proof}

\begin{proposition}
\label{prop:projective-mostly}
Assume~\eqref{eq:exp-moment} and
\(\lambda_1(A)>\lambda_2(A)\). Then the
restricted projective random map \(f|_X\) is mostly contracting where $X=P(\R^m)\setminus  P(E)$. 
More precisely, every $f$-stationary probability $\mu$ on $P(\R^m)$
with $\mu(P(E))=0$ satisfies
$\lambda(\mu)\leq\lambda_2(A)-\lambda_1(A)<0$.
\end{proposition}

\begin{proof}
If $\mu$ is an $f$-stationary probability on $P(\R^m)$ giving zero
mass to $P(E)$, then it is maximizing by
\cite[Cor.~5.4.2 and Rem.~5.4.3]{BMbook}. That is,  
\begin{equation}
\label{eq:maximizing-linear-stationary}
        \int \phi_A\,d\mu=\lambda_1(A), \quad \text{where} \quad \phi_A(\hat x)
 \eqdef
 \int\log\frac{\|A(\omega)x\|}{\|x\|}\,d\Pp.
\end{equation}
The additive cocycle identity and stationarity then give, for every
\(n\geq1\),
\begin{equation}
\label{eq:stationary-vector-growth}
 \int
 \log\frac{\|A^n(\omega)x\|}{\|x\|}
 \,d\Pp\,d\mu
 =
 n\lambda_1(A).
\end{equation}
The local projective Lipschitz estimate
\cite[Lem.~5.6.1]{BMbook} states that
\begin{equation*}
 Lf_\omega^n(\hat x)
 \leq
 \frac{\|\wedge^2A^n(\omega)\|}
 {\|A^n(\omega)x\|^2}
\end{equation*}
for every unit representative \(x\) of \(\hat x\). Taking logarithms,
integrating with respect to \(\Pp\) and \(\mu\), and
using~\eqref{eq:stationary-vector-growth}, we obtain
\[
 \frac1n
 \int\log Lf_\omega^n(\hat x)\,d\Pp\,d\mu
 \leq
 \frac1n\int\log\|\wedge^2A^n(\omega)\|\,d\Pp
 -2\lambda_1(A).
\]
Passing to the limit and using
\[
 \lim_{n\to \infty}
 \frac1n\int\log\|\wedge^2A^n(\omega)\|\,d\Pp
 =
 \lambda_1(A)+\lambda_2(A)
\]
proves~$\lambda(\mu)\leq \lambda_2(A)-\lambda_1(A)$ and concludes the proof.
\end{proof}

\begin{proposition}
\label{prop:projective-stationary-proximal}
Assume~\eqref{eq:exp-moment} and
\(\lambda_1(A)>\lambda_2(A)\). Then the projective random map
\(f|_X\) is stationarily proximal where $X=P(\R^m)\setminus P(E)$. 
\end{proposition}

\begin{proof}
For \(\hat x,\hat y\in X\), choose unit representatives \(x,y\). The angular distance can be written as 
\begin{equation}
\label{eq:linear-proximal-angular-bound}
 \begin{aligned}
 d(f_\omega^n(\hat x),f_\omega^n(\hat y))
 =
 \frac{\|\wedge^2A^n(\omega)(x\wedge y)\|}
 {\|A^n(\omega)x\|\,\|A^n(\omega)y\|}\leq
 \frac{\|\wedge^2A^n(\omega)\|}
 {\|A^n(\omega)x\|\,\|A^n(\omega)y\|}.
 \end{aligned}
\end{equation}
By \cite[Prop.~5.5.1 and Rem.~5.5.3]{BMbook}, every nonzero
vector \(z\notin E\) satisfies
\begin{equation}
\label{eq:linear-proximal-vector-growth}
 \lim_{n\to \infty}\frac1n\log\|A^n(\omega)z\|
 =\lambda_1(A)
 \qquad\text{for \(\Pp\)-a.e.~\(\omega\)}.
\end{equation}
On the other hand, Kingman's theorem applied to the exterior-square
cocycle gives
\begin{equation}
\label{eq:linear-proximal-wedge-growth}
 \lim_{n\to \infty}\frac1n
 \log\|\wedge^2A^n(\omega)\|
 =\lambda_1(A)+\lambda_2(A)
 \qquad\text{for \(\Pp\)-a.e.~\(\omega\)}.
\end{equation}
It follows from~\eqref{eq:linear-proximal-angular-bound}, \eqref{eq:linear-proximal-vector-growth}, and \eqref{eq:linear-proximal-wedge-growth} that, for every fixed
\(\hat x,\hat y\in X\) with $\hat x \neq \hat y$
\begin{equation*}
 \limsup_{n\to \infty}\frac1n
 \log d(f_\omega^n(\hat x),f_\omega^n(\hat y))
 \leq\lambda_2(A)-\lambda_1(A)<0 \quad \text{for \(\Pp\)-a.e.~\(\omega\).}
\end{equation*}
If $\hat x=\hat y$, the distance is identically zero. Thus
$d(f_\omega^n(\hat x),f_\omega^n(\hat y))\to 0$ for
$\Pp$-almost every~$\omega$, so $f|_X$ is synchronizing.

The implication \textup{(iii)}$\Rightarrow$\textup{(i)} of
Corollary~\ref{prop:synchronization-equivalences} now gives
stationary proximality.
\end{proof}

\begin{proof}[Proof of Theorem~\ref{main:spectral-gap-linear}]
If \(E=\{0\}\), then \(X=P(\R^m)\) is compact and the constant
function \(V\equiv1\) is a Lyapunov function. Propositions~\ref{prop:projective-mostly}
and~\ref{prop:projective-stationary-proximal}, together with
Theorem~\ref{thm:product-weighted-global-contraction}, give the result.
If \(E\neq\{0\}\), Proposition~\ref{prop:equator-repeller} verifies the
repeller and moment hypotheses, while Propositions~\ref{prop:projective-mostly}
and~\ref{prop:projective-stationary-proximal} verify the remaining
hypotheses of Corollary~\ref{cor:repeller-quasicompact}. 
The conclusion follows from that corollary.
\end{proof}

\subsection{Limit theorems}
We prove Theorems~\ref{main:norm-clt},
\ref{mainthm:pointwise-CLT-BE}, and~\ref{mainthm:linear-asip-final}.
For $\omega\in\Omega$ and $\hat x\in X$, write
\begin{equation}
\label{eq:linear-logarithmic-cocycle}
 \phi(\omega,\hat x)=\phi_\omega(\hat x)
 \eqdef
 \log\frac{\|A(\omega)x\|}{\|x\|},
 \qquad
 \xi(\omega,\hat x)
 \eqdef
 \phi(\omega,\hat x)-\lambda_1(A).
\end{equation}
These functions depend only on the first random symbol, and
\begin{equation}
\label{eq:linear-logarithmic-sum}
 \sum_{j=0}^{n-1}
 \xi\bigl(\sigma^j\omega,f_\omega^j(\hat x)\bigr)
 =
 \log\frac{\|A^n(\omega)x\|}{\|x\|}-n\lambda_1(A),
 \qquad n\geq1.
\end{equation}
Since
$\|A(\omega)^{-1}\|^{-1}\leq\|A(\omega)x\|/\|x\|
\leq\|A(\omega)\|$, we have
\begin{equation}
\label{eq:linear-logarithmic-one-step-bound}
 |\phi(\omega,\hat x)|\leq\log\ell(\omega),
 \qquad
 |\xi(\omega,\hat x)|
 \leq |\lambda_1(A)|+\log\ell(\omega).
\end{equation}
Moreover, by~\cite[Lem.~7.4.1]{BMbook}, for every
$0<\theta\leq1$ there is $K_\theta>0$ such that
\begin{equation}
\label{eq:linear-projective-potential-holder}
 |\phi(\omega,\hat x)-\phi(\omega,\hat y)|
 \leq
 K_\theta\ell(\omega)^{2\theta}d(\hat x,\hat y)^\theta.
\end{equation}

\subsubsection{Comparison with the operator norm}
We record the comparison estimate needed for the ASIP and for the indentification of variances;
compare~\cite[Claim~6.2.2]{BMbook}.

\begin{lemma}
\label{lem:linear-direction-comparison}

For every nonzero $x\in\R^m\setminus E$,
\begin{equation}\label{eq:linear-operator-remainder-as}
\log\frac{\|A^n(\omega)x\|}{\|x\|}
=\log\|A^n(\omega)\|+O_{x,\omega}(1)
\qquad\text{as $n\to \infty$, for $\Pp$-a.e.~$\omega$.}
\end{equation}

\end{lemma}

\begin{proof}
Fix $\hat x,\hat y\in X$. 
Choose $\theta,\beta>0$ such that $\beta<\varepsilon_0$,
$\theta/\beta<\varepsilon_0$ and $2\theta<\eta_0$, where
$\varepsilon_0$ is given by Theorem~\ref{main:spectral-gap-linear}
and $\eta_0$ by~\eqref{eq:exp-moment}. Set
$X_i=f_\omega^i(\hat x)$, $Y_i=f_\omega^i(\hat y)$ and
$D_i=\phi(\sigma^i\omega,Y_i)-\phi(\sigma^i\omega,X_i)$.
By~\eqref{eq:linear-projective-potential-holder}, independence of
$(X_i,Y_i)$ and $\omega_i$, and mean contraction,
\[
 \int |D_i|\,d\Pp
 \leq K_\theta\left(\int\ell(\omega)^{2\theta}\,d\Pp\right)
       \int d(X_i,Y_i)^\theta\,d\Pp
 \leq C\rho^i\Delta_{\theta,\beta}(\hat x,\hat y).
\]
Here $V\geq1$ and the moment is finite by~\eqref{eq:exp-moment}.

Hence $\sum_{i\geq0}\int|D_i|\,d\Pp<\infty$.
By Tonelli and the cocycle identity, for each fixed pair
$\hat x,\hat y\in X$,
\[
\sup_{n\geq1}\left|
\log\frac{\|A^n(\omega)y\|}{\|y\|}
-\log\frac{\|A^n(\omega)x\|}{\|x\|}
\right|\leq\sum_{i\geq0}|D_i(\omega)|<\infty
\quad\text{$\Pp$-a.s.}
\]

Now fix $x\notin E$ and choose a basis $x_1,\ldots,x_m$ of unit
vectors outside $E$ with $x_1=x/\|x\|$. For some $C_{\mathcal B}\geq1$,
every linear map $B:\R^m\to\R^m$ satisfies
\begin{equation}\label{eq:linear-basis-norm-equivalence}
\max_{1\leq j\leq m}\|Bx_j\|\leq\|B\|
\leq C_{\mathcal B}\max_{1\leq j\leq m}\|Bx_j\|.
\end{equation}
Consequently,
\[
0\leq\log\|A^n(\omega)\|-\log\|A^n(\omega)x_1\|
\leq\log C_{\mathcal B}
+\max_{1\leq j\leq m}\left|\log\|A^n(\omega)x_j\|
-\log\|A^n(\omega)x_1\|\right|.
\]
The pairwise comparison applies to this finite family and proves
\eqref{eq:linear-operator-remainder-as}.
\end{proof}

\subsubsection{Proof of Theorem~\ref{mainthm:pointwise-CLT-BE}}

Choose admissible $\alpha,\beta>0$ and $\eta>0$ so small that
$\eta+4\alpha<\eta_0$. Apply
Theorem~\ref{thm:abstract-limit-theorems} to the one-step observable
$u=\phi$ from~\eqref{eq:linear-logarithmic-cocycle}. Indeed,
\eqref{eq:linear-logarithmic-one-step-bound} and
\eqref{eq:linear-projective-potential-holder} give
\[
 \int (1+[\phi_\omega]_\alpha)^2
 e^{\eta\|\phi_\omega\|_\infty}\,d\Pp
 \leq C\int\ell(\omega)^{\eta+4\alpha}\,d\Pp<\infty,
\]
so~\eqref{eq:observable-mixed-moment} holds. Moreover,
\eqref{eq:maximizing-linear-stationary} identifies the stationary mean
of $\phi$ with $\lambda_1(A)$. By~\eqref{eq:linear-logarithmic-sum},
the Birkhoff sum satisfies
$S_n(\omega,\hat x)=\log(\|A^n(\omega)x\|/\|x\|)$.
Items~\textup{(i)}--\textup{(ii)} of
Theorem~\ref{thm:abstract-limit-theorems} give the
pointwise CLT and Berry--Esseen estimate, with variance $\sigma_\phi^2$, associated with the one-step $u=\phi$. We first prove the assertions with $\sigma_\phi$. The proof of Theorem~\ref{main:norm-clt} below identifies this variance
with the  variance $\sigma^2$ in the CLT for $\log \|A^n(\omega)\|$.
Theorem~\ref{thm:abstract-limit-theorems} also gives the local limit and renewal assertions whenever
$\phi$ is non-arithmetic in the sense of
Definition~\ref{def:observable-nonarithmetic} and Proposition~\ref{prop:measurable-arithmeticity}. Finally, item~\textup{(v)} of Theorem~\ref{thm:abstract-limit-theorems} gives the large deviation rates when $\sigma_\phi>0$ and the
exponential upper bound for every $\sigma_\phi\geq0$.

\subsubsection{Proof of Theorem~\ref{main:norm-clt}}
We use the tensor product $\R^m\otimes\R^m$ and the associated
Kronecker product of matrices; see~\cite[Sec.~4.2]{HornJohnsonTopics}.
Recall that the vectors
$v\otimes w$ span $\R^m\otimes\R^m$, and
\[
 (B\otimes C)(v\otimes w)=Bv\otimes Cw,
 \qquad
 \langle v_1\otimes w_1,v_2\otimes w_2\rangle
 =
 \langle v_1,v_2\rangle\langle w_1,w_2\rangle.
\]
Set $A^{(2)}(\omega)\eqdef A(\omega)\otimes A(\omega)$.
Then $(A^{(2)})^n(\omega)
 =
 A^n(\omega)\otimes A^n(\omega)$
and $\|B\otimes C\|=\|B\|\,\|C\|$. Hence $A^{(2)}$ satisfies the
exponential moment condition~\eqref{eq:exp-moment}, with exponent
$\eta_0/2$.

Let $s_1(B)\geq\cdots\geq s_m(B)>0$ denote the singular values of $B$.
By~\cite[Thm.~4.2.15]{HornJohnsonTopics}, the singular values of
$B\otimes B$ are the products $s_i(B)s_j(B)$. Thus the two largest
ones are $s_1(B)^2$ and $s_1(B)s_2(B)$. Applying this to
$B=A^n(\omega)$ and using the definition of the Lyapunov exponents,
we obtain
\[
 \lambda_1(A^{(2)})=2\lambda_1(A),
 \qquad
 \lambda_2(A^{(2)})
 =
 \lambda_1(A)+\lambda_2(A).
\]
In particular, $A^{(2)}$ also has a simple top Lyapunov exponent.

Fix an orthonormal basis $e_1,\ldots,e_m$ of $\R^m$ and consider the
unit vector
\[
 u\eqdef\frac1{\sqrt m}\sum_{j=1}^m e_j\otimes e_j.
\]
For every matrix $B$, using the tensor-product inner product and then
the singular values, 
\begin{align*}
 \|(B\otimes B)u\|^2=
 \frac1m
 \sum_{i,j=1}^m
 \langle Be_i,Be_j\rangle^2  =
 \frac1m\operatorname{tr}\bigl((B^{\mathsf T}B)^2\bigr)
 =
 \frac1m\sum_{i=1}^m s_i(B)^4.
\end{align*}
Since $s_1(B)=\|B\|$, it follows that
\begin{equation}
\label{eq:tensor-norm-two-sided}
 \frac1{\sqrt m}\|B\|^2
 \leq
 \|(B\otimes B)u\|
 \leq
 \|B\|^2
\end{equation}
and therefore
\begin{equation}
\label{eq:tensor-norm-comparison}
 \left|
 \frac12\log\|(B\otimes B)u\|-\log\|B\|
 \right|
 \leq\frac14\log m.
\end{equation}
Taking $B=A^n(\omega)$ in~\eqref{eq:tensor-norm-two-sided} and using
the Furstenberg--Kesten theorem~\cite{FurstenbergKestenCLT} gives
\[
 \lim_{n\to \infty}\frac1n
 \log\|(A^{(2)})^n(\omega)u\|
 =
 2\lambda_1(A)
 =
 \lambda_1(A^{(2)})
 \qquad\Pp\text{-a.s.}  \enlargethispage{2cm}
\]
Hence $\hat u$ does not belong to the equator of $A^{(2)}$, and
Theorem~\ref{mainthm:pointwise-CLT-BE} applies to $A^{(2)}$ and this
fixed direction. Let $\tau^2$ be the variance given by that theorem and put
$\sigma\eqdef\tau/2$.

Set
$U_n=\frac12\log\|(A^{(2)})^n(\omega)u\|-n\lambda_1(A)$ and
$V_n=\log\|A^n(\omega)\|-n\lambda_1(A)$. By
\eqref{eq:tensor-norm-comparison},
$|U_n-V_n|\leq\frac14\log m$. The pointwise central limit theorem
for $A^{(2)}$ gives $U_n/\sqrt n\to\cN(0,\sigma^2)$ in law;
thus, when $\sigma=0$, Slutsky's theorem proves the asserted CLT for
$V_n$.

When $\sigma>0$, let $\Phi$ be the standard normal distribution
function. The pointwise Berry--Esseen bound for $A^{(2)}$, together
with the comparison above and the fact that $\Phi$ is
$1/\sqrt{2\pi}$-Lipschitz, gives
\[
\sup_{s\in\R}\left|
\Pp\left(\frac{V_n}{\sigma\sqrt n}\leq s\right)-\Phi(s)
\right|
\leq \frac C{\sqrt n}
+\frac{\log m}{4\sigma\sqrt{2\pi n}}.
\]
This proves both the CLT and~\eqref{eq:norm-BE} in this case.

The variance $\sigma^2$ obtained here agrees with the pointwise variance $\sigma_\phi^2$ obtained in the preceding proof of pointwise CLT in 
Theorem~\ref{mainthm:pointwise-CLT-BE}. 
Indeed, for each fixed $x\notin E$,
\eqref{eq:linear-operator-remainder-as} makes the difference between
the two centered logarithms almost surely $O(1)$. After division by
$\sqrt n$ it tends to zero, Slutsky’s theorem proves that both CLTs have the same variance, i.e., $\sigma^2=\sigma_\phi^2$.

Finally, write $M=\frac14\log m$. For $n\varepsilon>2M$,
\eqref{eq:tensor-norm-comparison} and the pointwise bound in
Theorem~\ref{mainthm:pointwise-CLT-BE}\textup{(v)}, applied to
$A^{(2)}$ and $u$, give
$\Pp(|V_n|>n\varepsilon)
\leq\Pp(|U_n|>n\varepsilon/2)
\leq C e^{-h n\varepsilon^2}$. 
For $n\varepsilon\leq2M$, we have
$n\varepsilon^2\leq2M\varepsilon_0$. Thus, enlarging $C$
so that $C\geq e^{2hM\varepsilon_0}$ gives
$\Pp(|V_n|>n\varepsilon)\leq1\leq Ce^{-hn\varepsilon^2}$.  This proves~\eqref{eq:LDa}
also when $\sigma=0$.

Suppose that $\sigma>0$, and let $c$ be the rate function in
Theorem~\ref{mainthm:pointwise-CLT-BE}\textup{(v)}. In addition to
the upper bound in Theorem~\ref{main:norm-clt}\textup{(iii)}, one has
exact logarithmic rates for both tails. Indeed, the finite-direction
comparison of~\cite[proof of Thm.~6.2.1]{BMbook}, applied with a
basis of vectors outside $E$, gives,
for $0<\varepsilon<\varepsilon_0$,
\begin{equation}\label{eq:norm-positive-LD}
 \lim_{n\to \infty}\frac1n\log
 \Pp\left(
 \frac1n\log\|A^n(\omega)\|-\lambda_1(A)>\varepsilon
 \right)
 =-c(\varepsilon).
\end{equation}
To obtain the lower tail, consider the tensor-square cocycle
$A^{(2)}=A\otimes A$ and the vector $u$ introduced above. Recall from~\eqref{eq:tensor-norm-two-sided} that
\[
 \left|
 \frac12\log\|(A^{(2)})^n(\omega)u\|
 -\log\|A^n(\omega)\|
 \right|
 \leq\frac14\log m.
\]
Let $c_2$ be the local rate function associated with $A^{(2)}$ and
$u$. Since $\lambda_1(A^{(2)})=2\lambda_1(A)$ and its asymptotic
variance is $4\sigma^2>0$, the preceding pointwise argument gives
\[
 \lim_{n\to \infty}\frac1n\log
 \Pp\left(
 \frac1n\log\|A^n(\omega)\|-\lambda_1(A)>\varepsilon
 \right)
 =-c_2(2\varepsilon)
\]
and, similarly,
\[
 \lim_{n\to \infty}\frac1n\log
 \Pp\left(
 \frac1n\log\|A^n(\omega)\|-\lambda_1(A)<-\varepsilon
 \right)
 =-c_2(-2\varepsilon).
\]
Together with~\eqref{eq:norm-positive-LD}, this yields
$c_2(2\varepsilon)=c(\varepsilon)$ for all sufficiently small
$\varepsilon>0$. The logarithms of the leading eigenvalues of the
Laplace operators are real analytic near zero, and their second
derivatives there are respectively $\sigma^2>0$ and
$4\sigma^2>0$. The inverse function theorem  shows that
their local Legendre transforms $c$ and $c_2$ are real analytic near
zero. By the identity theorem, $c_2(2\varepsilon)=c(\varepsilon)$
for all small $|\varepsilon|$. The lower-tail rate is therefore $c(-\varepsilon)$, and combining the two tails proves the
corresponding two-sided logarithmic rates for the operator norm.

\subsubsection{Proof of Theorem~\ref{mainthm:linear-asip-final}}

We apply Theorem~\ref{cor:ASIP-global-contraction} to the one-step
potential $\phi$  in~\eqref{eq:linear-logarithmic-cocycle}.  For the same sufficiently
small $\theta,\beta>0$ chosen in Lemma~\ref{lem:linear-direction-comparison},
\eqref{eq:linear-projective-potential-holder} and $V\geq1$ give
\[
 \int
 |\phi(\omega,\hat x)-\phi(\omega,\hat y)|\,d\Pp
 \leq
 K_\theta
 \left(\int\ell(\omega)^{2\theta}\,d\Pp\right)
 \Delta_{\theta,\beta}(\hat x,\hat y).
\]
Thus~\eqref{eq:asip-final-observable-regularity} holds.  Moreover, $|\phi(\omega,\hat x)|\leq\log\ell(\omega)$,
so~\eqref{eq:exp-moment} implies $\int e^{a|\phi|}\,d\Pp\,d\nu<\infty$ for some $a>0$.
Finally, \eqref{eq:maximizing-linear-stationary} identifies the
stationary mean of $\phi$ with $\lambda_1(A)$.
Write $\sigma_\phi^2$ for the ASIP variance obtained by
applying Theorem~\ref{cor:ASIP-global-contraction} to $u=\phi$.  By Remark~\ref{rem:variance_ASIP}, $\sigma^2_\phi$ coincides with the  variance $\sigma^2$ in the
pointwise CLT.  Hence that theorem yields, for every fixed
$x\in\R^m\setminus E$,
\[
 \log\frac{\|A^n(\omega)x\|}{\|x\|}
 -n\lambda_1(A)-\sum_{k=1}^nG_k
 =
 O((\log n)^3)
 \qquad\text{a.s.}
\]
for independent $G_k\sim\cN(0,\sigma_\phi^2)$. Since $(\log n)^3=o(\sqrt n)$, this ASIP implies the pointwise CLT with
variance $\sigma_\phi^2$.   For the operator norm, the two processes are coupled along the same matrix
product and their difference is almost surely bounded by
\eqref{eq:linear-operator-remainder-as}. Lemma~\ref{lem:asip-bounded-comparison}
therefore transfers the ASIP with the same variance and error to
$\log\|A^n(\omega)\|$. Comparing with
the qualitative CLT in
Theorem~\ref{main:norm-clt} also identifies this variance with the one
appearing there.

If the law of $A(\omega)$ has compact support, then $\ell$ is bounded,
and hence so is $\phi$.  Theorem~
\ref{cor:ASIP-global-contraction} then applies with $\gamma=\infty$ and
gives the improved error $O((\log n)^2)$, both pointwise and, by the
same bounded comparison, for the operator norm.  This completes the
proof.

\subsubsection{Positivity of the variance}

We prove Proposition~\ref{prop:linear-zero-variance}.
The proof of
Theorem~\ref{mainthm:linear-asip-final}, together with
\eqref{eq:linear-operator-remainder-as}, shows that the pointwise and
operator-norm variances coincide.

Let $Q:\R^m\to\tilde V$ be the quotient map and, for
$\hat x\in X$, put
$b(\hat x)\eqdef\log\frac{\|Qx\|}{\|x\|}$.
For the angular metric and the quotient Euclidean norm,
$\|Qx\|/\|x\|=\delta(\hat x)$, so $b=\log\delta\leq0$.  Since
$QA(\omega)=\tilde A(\omega)Q$, summing the one-step identity gives
\begin{equation}
\label{eq:linear-quotient-telescoping}
 \log\frac{\|A^n(\omega)x\|}{\|x\|}
 -\log\frac{\|\tilde A^n(\omega)Qx\|}{\|Qx\|}
 =b(\hat x)-b(f_\omega^n(\hat x)).
\end{equation}
If $E=\{0\}$, then $b=0$. Otherwise, let $V$ be the Lyapunov
function furnished by Proposition~\ref{prop:repeller-drift}. Since
$\delta^{-\gamma}\leq V$ for some $\gamma>0$ and
$(\log(1/s))^2\leq C_\gamma s^{-\gamma}$ for $0<s\leq1$,
$|b|^2\leq C_\gamma V$. Proposition~\ref{prop:tightness} gives
$\int V\,d\nu<\infty$. In either case, stationarity yields
\[
 \sup_{n\geq0}
 \int|b(\hat x)-b(f_\omega^n(\hat x))|^2\,d\Pp\,d\nu
 \leq4\int b^2\,d\nu<\infty.
\]
Thus~\eqref{eq:linear-quotient-telescoping} shows that the stationary
logarithmic cocycle of $A$ and the quotient cocycle, started from the
pushforward of $\nu$, have the same asymptotic variance.  This is precisely the asymptotic variance $\sigma^2$ appearing in Theorem~\ref{mainthm:pointwise-CLT-BE}.  By
\cite[Claim~5.5.2]{BMbook}, $\tilde A$ is quasi-irreducible,
$\lambda_1(\tilde A)=\lambda_1(A)$ and
$\lambda_2(\tilde A)\leq\lambda_2(A)$.  If $\dim\tilde V=1$, its variance is
$\Var(\log\|\tilde A(\omega)\|)$, so the asserted equivalence follows
directly. If $\dim\tilde V\geq2$,
\cite[Prop.~IX]{BMbook} applied to $\tilde A$ proves the proposition.


\section{Random interval maps}
\label{s:interval}

We conclude with nonlinear applications illustrating that the abstract theory applies when the common invariant set consists of boundary points.

\subsection{Random interval diffeomorphisms}
Let $f_1,\ldots,f_k\in\mathrm{Diff}^2_+([0,1])$ have probabilities
$p_i>0$, with $\sum_{i=1}^kp_i=1$. Let $f$ be the associated random
map and $P$ its Markov operator. Every $f_i$ fixes $0$ and $1$, so
$E=\{0,1\}$ is totally invariant. Assume:
\begin{enumerate}[itemsep=0.1cm]
\item[\rm(i)] for every $x\in(0,1)$ there are $i,j$ such that $f_i(x)<x<f_j(x)$;
\item[\rm(ii)] $\chi_i\eqdef \sum_{j=1}^kp_j\log f_j'(i)>0$ for $i=0,1$.
\end{enumerate}
These are the hypotheses of~\cite{CzudekInterval}. For two maps
with $f_1(x)<x<f_2(x)$, the qualitative results used below were
proved in~\cite{GharaeiHomburg}. Put
$\delta(x)=d(x,E)=\min\{x,1-x\}$, and equip the weighted
H\"older space $\Hol_{\alpha,\beta}((0,1))$ with the norm
\[
 \|\varphi\|_{\alpha,\beta}=|\varphi(x_0)|+
 \sup_{x\neq y}\frac{|\varphi(x)-\varphi(y)|\delta(x)^\beta\delta(y)^\beta}{|x-y|^\alpha}.
\]

\begin{corollary}\label{cor:interval-spectral-gap}
Under \textup{(i)}--\textup{(ii)}, there is $\varepsilon_0>0$ such that,
whenever $\alpha,\beta>0$ satisfy $\beta<\varepsilon_0$ and
$\alpha/\beta<\varepsilon_0$, there are $C>0$ and
$\rho\in(0,1)$ such that
\begin{align*}
\int
 \frac{
  |f_\omega^n(x)-f_\omega^n(y)|^\alpha
 }{
  \delta(f_\omega^n(x))^\beta
  \delta(f_\omega^n(y))^\beta
 }\,
 d\Pp \leq
 C\rho^n
 \frac{
  |x-y|^\alpha
 }{
  \delta(x)^\beta\delta(y)^\beta } \qquad \text{for every \(x,y\in (0,1)\) and \(n\geq0\).
 }
\end{align*}
Moreover, $P$ has a unique stationary probability $\nu$ on $(0,1)$ and a spectral gap on $\Hol_{\alpha,\beta}((0,1))$. Every bounded one-step observable uniformly H\"older in $x$ satisfies, for each fixed $x\in(0,1)$, an ASIP with error $O((\log n)^2)$. In particular, for some $\sigma^2\geq0$ and independent $G_j\sim\cN(0,\sigma^2)$,
\[
 \log(f_\omega^n)'(x)
 -n\sum_{i=1}^kp_i\int\log f_i'\,d\nu
 -\sum_{j=1}^nG_j
 =O((\log n)^2)\qquad\text{a.s.}
\]
\end{corollary}

\begin{proof}
The moment hypothesis of Proposition~\ref{prop:repeller-drift} is automatic for this finite family of $C^2$ diffeomorphisms. Near $0$ and $1$, respectively,
\[
 \psi_{1}(0)=\lim_{x\to 0^+}\sum_{i=1}^kp_i\log\frac{x}{f_i(x)}=-\chi_0<0
 \  \text{and} \   
 \psi_{1}(1)=\lim_{x\to 1^-}\sum_{i=1}^kp_i\log\frac{1-x}{1-f_i(x)}=-\chi_1<0.
\]
Thus $E$ is repelling already for one iterate. By \cite[Sec.~2.1 and Thm.~1]{CzudekInterval}, assumptions (i)--(ii) also give a unique stationary probability $\nu$ on $(0,1)$, negative Lyapunov exponent
\[
        \lambda(\nu)=\sum_{i=1}^kp_i\int\log f_i'\,d\nu<0,
\]
and almost-sure synchronization of every pair of interior points. Hence $f|_X$ is mostly contracting and, by Corollary~\ref{prop:synchronization-equivalences}, stationarily proximal. Corollary~\ref{cor:repeller-quasicompact} gives the global contraction and spectral gap.

If $\phi$ is bounded and uniformly $\eta$-H\"older in $x$, choose the admissible $\alpha\leq\eta$. Since $V\geq1$ and $d\leq1$,
$\int|\phi(\omega,x)-\phi(\omega,y)|\,d\Pp\leq L|x-y|^\eta\leq L\Delta_{\alpha,\beta}(x,y)$. Theorem~\ref{cor:ASIP-global-contraction} with $\gamma=\infty$ gives the $O((\log n)^2)$ ASIP. Finally, each $\log f_i'$ is bounded and Lipschitz, and the chain rule gives $\sum_{j=0}^{n-1}\log f_{\omega_j}'(f_\omega^j(x))=\log(f_\omega^n)'(x)$, proving the last assertion.
\end{proof}

\subsection{Random logistic maps}
Consider the iid random logistic maps
\[
        f_\omega(x)=c(\omega)x(1-x),\qquad x\in[0,1],
\]
with one-step $c(\omega)>0$, i.e., $c(\omega)=c_{\omega_0}$. Assume that, for some $c_+<3$ and $\eta_0>0$,
\begin{equation}\label{eq:logistic-assumptions}
 0<c(\omega)\leq c_+\quad\Pp\text{-a.s.},\qquad
 \int c(\omega)^{-\eta_0}\,d\Pp<\infty,
 \qquad
 \chi\eqdef \int\log c(\omega)\,d\Pp>0.
\end{equation}
Set $M=[0,1]$, $E=\{0,1\}$, $X=(0,1)$, and $\delta(x)=d(x,E)=\min\{x,1-x\}$. The condition $\chi>0$ is the classical supercritical condition.

\begin{corollary}\label{cor:random-logistic}
Under \eqref{eq:logistic-assumptions}, the global-contraction and spectral-gap conclusions of Corollary~\ref{cor:interval-spectral-gap} hold. Moreover, every bounded one-step observable uniformly H\"older in $x$ satisfies, for each fixed $x\in(0,1)$, an ASIP with error $O((\log n)^2)$.
\end{corollary}

\begin{proof}
Since $c_+<4$, $0<f_\omega(x)\leq c_+/4<1$ on $X$.
As $f_\omega(x)=0$ exactly at $E$, we have $f_\omega^{-1}(E)=E$.
Also, $\Lip(f_\omega)\leq c_+$ and
$x(1-x)\geq\delta(x)/2$. If $f_\omega(x)\leq1/2$, then
$\delta(f_\omega(x))=f_\omega(x) 
 \geq c(\omega)\delta(x)/2$, otherwise $\delta(f_\omega(x))=1-f_\omega(x) \geq1-c_+/4$ while $\delta(x)\leq1/2$. Thus
\[
 K(\omega)=\sup_{x\in (0,1)} \frac{\delta(x)}{\delta(f_\omega(x))} \leq\max\left\{\frac2{c(\omega)},\frac2{4-c_+}\right\},
\]
so the moment hypothesis of Proposition~\ref{prop:repeller-drift} follows from~\eqref{eq:logistic-assumptions}.
Near $0$ and $1$, the expression
$\log(\delta(x)/\delta(f_\omega(x)))$ equals, respectively,
$-\log c(\omega)-\log(1-x)$ and
$-\log c(\omega)-\log x$. Thus
$\psi_{1}(0)=\psi_{1}(1)=-\int \log c(\omega)\,d\Pp= -\chi<0$,  and $E$ is repelling according to Definition~\ref{def:repelling-set}.

For the interior dynamics, Steinsaltz \cite[pp.~559--560]{SteinsaltzLogistic} recalls Kl\"unger's attractivity result for parameters in $[0,3]$ and that the Lyapunov exponent is negative when the parameters are strictly below $3$. Thus there is a unique stationary probability on $X$ and $f|_X$ is mostly contracting. Here attractivity means almost-sure convergence to a common limit for the reversed products $G_n=f_{\omega_0}\circ\cdots\circ f_{\omega_{n-1}}$ on $X$. Since the iid law is invariant under reversal, $G_n\stackrel{\rm law}=f_\omega^n$ for each $n$; bounded convergence of the distance between the images of any two interior points therefore gives synchronization in mean for the forward products, and Corollary~\ref{prop:synchronization-equivalences} yields stationary proximality. Corollary~\ref{cor:repeller-quasicompact} now gives global contraction and the spectral gap.

The ASIP follows exactly as above: a bounded uniformly $\eta$-H\"older one-step observable satisfies \eqref{eq:asip-final-observable-regularity} for admissible $\alpha\leq\eta$, and Theorem~\ref{cor:ASIP-global-contraction} with $\gamma=\infty$ gives error $O((\log n)^2)$.
\end{proof}

\begin{remark}\label{rem:logistic-beyond-three}
The restriction $c_+<3$ only ensures attractivity without further assumptions. More general criteria allowing parameters above $3$ are given in \cite[Thm.~2]{SteinsaltzLogistic}. In particular, \cite[Cor.~3]{SteinsaltzLogistic} gives attractive iid examples with $c/4\sim{\rm Beta}(a+\tfrac12,a-\tfrac12)$, $a>1$, whose support is $(0,4)$. In this case
$K(\omega)\leq\max\{2/c(\omega),2/(4-c(\omega))\}$, and the Beta endpoint moments give the moment hypothesis of Proposition~\ref{prop:repeller-drift}. Hence the same weighted spectral gap and $O((\log n)^2)$ ASIP hold with parameters arbitrarily close to $0$ and $4$.
\end{remark}

\section*{Acknowledgments}

P.~G.~Barrientos was partially supported by the Spanish MCIN under
grant PID-2023-147461NB-I00, by CNPq under grants
PQ~302738/2025-8 and Projeto Universal~401737/2025-0,
and by CAPES---Finance Code~001.
D.~Malicet acknowledges support from the
R\'eseau Franco-Br\'esilien de Math\'ematiques (RFBM)
for a research visit to IMPA and UFF in August 2025, when this paper was developed.

\paragraph{\bf Declaration of AI use.} AI tools were used to assist with English-language editing
and the clarification and refinement of some proofs and applications.
The proof of the main theorem was developed without AI assistance.
The authors reviewed the resulting text and take full responsibility
for the mathematical content and the final manuscript.

\bibliographystyle{alpha3}
\bibliography{biblio}

@article{AG20,
  author  = {Aoun, Richard and Guivarc'h, Yves},
  title   = {Random matrix products when the top {L}yapunov exponent is simple},
  journal = {Journal of the European Mathematical Society},
  volume  = {22},
  number  = {7},
  pages   = {2135--2182},
  year    = {2020},
  doi     = {10.4171/JEMS/962}
}

@book{BMbook,
  author    = {Barrientos, Pablo G. and Malicet, Dominique},
  title     = {Mostly Contracting Random Maps},
  series    = {Lecture Notes in Mathematics},
  volume    = {2402},
  publisher = {Springer},
  address   = {Cham},
  year      = {2026},
  note      = {Forthcoming},
  isbn      = {978-3-032-35290-3}
}

@book{BenaimHurth,
  author    = {Bena{\"i}m, Michel and Hurth, Tobias},
  title     = {Markov Chains on Metric Spaces: A Short Course},
  series    = {Universitext},
  publisher = {Springer},
  address   = {Cham},
  year      = {2022},
  doi       = {10.1007/978-3-031-11822-7},
  isbn      = {978-3-031-11821-0}
}

@article{BenoistQuint,
  author  = {Benoist, Yves and Quint, Jean-Fran{\c c}ois},
  title   = {Central Limit Theorem for Linear Groups},
  journal = {The Annals of Probability},
  volume  = {44},
  number  = {2},
  pages   = {1308--1340},
  year    = {2016},
  doi     = {10.1214/15-AOP1002}
}

@article{CDM18,
  author  = {Cuny, Christophe and Dedecker, J{\'e}r{\^o}me and Merlev{\`e}de, Florence},
  title   = {On the {K}oml{\'o}s, {M}ajor and {T}usn{\'a}dy strong approximation for some classes of random iterates},
  journal = {Stochastic Processes and their Applications},
  volume  = {128},
  number  = {4},
  pages   = {1347--1385},
  year    = {2018},
  doi     = {10.1016/j.spa.2017.07.011}
}

@article{CDM24,
  author  = {Cuny, Christophe and Dedecker, J{\'e}r{\^o}me and Merlev{\`e}de, Florence},
  title   = {Strong Approximations for a Class of Dependent Random Variables with Semi-Exponential Tails},
  journal = {Journal of Theoretical Probability},
  volume  = {37},
  number  = {3},
  pages   = {2234--2252},
  year    = {2024},
  doi     = {10.1007/s10959-023-01306-0}
}

@article{GuivarchUrban05,
  author  = {Guivarc'h, Yves and Urban, Roman},
  title   = {Semigroup actions on tori and stationary measures
             on projective spaces},
  journal = {Studia Mathematica},
  volume  = {171},
  number  = {1},
  pages   = {33--66},
  year    = {2005},
  doi     = {10.4064/sm171-1-3}
}

@article{GR85,
  title={Frontiere de {Furstenberg}, propri{\'e}t{\'e}s de contraction et th{\'e}oremes de convergence},
  author={Guivarc'h, Yves and Raugi, Albert},
  journal={Zeitschrift f{\"u}r Wahrscheinlichkeitstheorie und Verwandte Gebiete},
  volume={69},
  pages={187--242},
  year={1985},
  publisher={Springer}
}

@article{FK83,
  title={Random matrix products and measures on projective spaces},
  author={Furstenberg, Harry and Kifer, Yuri},
  journal={Israel Journal of Mathematics},
  volume={46},
  pages={12--32},
  year={1983},
  publisher={Springer}
}

@misc{CzudekInterval,
  author        = {Czudek, Klaudiusz},
  title         = {Exponential decay of correlations for random interval diffeomorphisms},
  year          = {2026},
  eprint        = {2504.16549},
  archivePrefix = {arXiv},
  primaryClass  = {math.DS},
  doi           = {10.48550/arXiv.2504.16549},
  note          = {arXiv:2504.16549v2}
}

@article{FurstenbergKestenCLT,
  author  = {Furstenberg, Harry and Kesten, Harry},
  title   = {Products of Random Matrices},
  journal = {The Annals of Mathematical Statistics},
  volume  = {31},
  number  = {2},
  pages   = {457--469},
  year    = {1960},
  doi     = {10.1214/aoms/1177705909}
}

@article{GharaeiHomburg,
  author  = {Gharaei, Masoumeh and Homburg, Ale Jan},
  title   = {Random interval diffeomorphisms},
  journal = {Discrete and Continuous Dynamical Systems -- Series S},
  volume  = {10},
  number  = {2},
  pages   = {241--272},
  year    = {2017},
  doi     = {10.3934/dcdss.2017012}
}

@book{HH01,
  author    = {Hennion, Hubert and Herv{\'e}, Lo{\"i}c},
  title     = {Limit Theorems for {M}arkov Chains and Stochastic Properties of Dynamical Systems by Quasi-Compactness},
  series    = {Lecture Notes in Mathematics},
  volume    = {1766},
  publisher = {Springer},
  address   = {Berlin, Heidelberg},
  year      = {2001},
  doi       = {10.1007/b87874},
  isbn      = {978-3-540-42415-4}
}

@book{HornJohnsonTopics,
  author    = {Horn, Roger A. and Johnson, Charles R.},
  title     = {Topics in Matrix Analysis},
  publisher = {Cambridge University Press},
  address   = {Cambridge},
  year      = {1991},
  doi       = {10.1017/CBO9780511840371},
  isbn      = {978-0-521-46713-1}
}

@incollection{LePage,
  author    = {Le Page, {\'E}mile},
  title     = {Th{\'e}or{\`e}mes limites pour les produits de matrices al{\'e}atoires},
  booktitle = {Probability Measures on Groups},
  editor    = {Heyer, Herbert},
  series    = {Lecture Notes in Mathematics},
  volume    = {928},
  pages     = {258--303},
  publisher = {Springer},
  address   = {Berlin, Heidelberg},
  year      = {1982},
  doi       = {10.1007/BFb0093229}
}

@article{Mal17,
  author  = {Malicet, Dominique},
  title   = {Random Walks on {Homeo}($S^1$)},
  journal = {Communications in Mathematical Physics},
  volume  = {356},
  number  = {3},
  pages   = {1083--1116},
  year    = {2017},
  doi     = {10.1007/s00220-017-2996-5}
}

@article{SteinsaltzLogistic,
  author  = {Steinsaltz, David},
  title   = {Random Logistic Maps and Lyapunov Exponents},
  journal = {Indagationes Mathematicae},
  volume  = {12},
  number  = {4},
  pages   = {557--584},
  year    = {2001},
  doi     = {10.1016/S0019-3577(01)80042-2}
}

@article{walkden2013transfer,
  author  = {Walkden, Charles},
  title   = {Transfer operators for contractive {M}arkov systems and stochastic stability of the invariant measure},
  journal = {Dynamical Systems},
  volume  = {28},
  number  = {1},
  pages   = {34--47},
  year    = {2013},
  doi     = {10.1080/14689367.2012.739598}
}

@book{BouLac:85,
  author    = {Bougerol, Philippe and Lacroix, Jean},
  title     = {Products of Random Matrices with Applications to {Schr{\"o}dinger} Operators},
  series    = {Progress in Probability and Statistics},
  volume    = {8},
  publisher = {Birkh{\"a}user},
  address   = {Boston, MA},
  year      = {1985},
  doi       = {10.1007/978-1-4684-9172-2},
  isbn      = {978-1-4684-9174-6}
}

\end{document}